\documentclass{amsart}
\usepackage{latexsym,amscd,amssymb,url}
\usepackage{extarrows}
\usepackage{verbatim}
\usepackage{mathtools}
\usepackage{bbm}
\usepackage{dsfont}
\usepackage{tikz-cd}

\usepackage[all,cmtip]{xy}  
\usepackage{graphicx}
\usepackage{xcolor}
\usepackage{enumitem} 
\definecolor{dblue}{rgb}{0,0,.6}
\usepackage[colorlinks=true, linkcolor=dblue, citecolor=dblue, filecolor = dblue, menucolor = dblue, urlcolor = dblue]{hyperref}

\numberwithin{equation}{section}

\newtheorem{theorem}{Theorem}[section]

\theoremstyle{plain}

\newtheorem{conjecture}[theorem]{Conjecture}
\newtheorem{corollary}[theorem]{Corollary}

\newtheorem{lemma}[theorem]{Lemma}

\newtheorem{remark}[theorem]{Remark}

\newcommand{\del}{\partial}
\newcommand{\Z}{\mathbb Z}
\newcommand{\Q}{\mathbb Q}

\newcommand{\C}{\mathbb C}

\newcommand{\R}{\operatorname{R}}
\newcommand{\HH}{H}

\newcommand{\A}{A}
\newcommand{\Co}{K}

\newcommand{\im}{\operatorname{im}}

\newcommand{\Hom}{\operatorname{Hom}}

\newcommand{\id}{\operatorname{id}}

\newcommand{\Spec}{\operatorname{Spec}}

\newcommand{\codim}{\operatorname{codim}}

\newcommand{\Br}{\operatorname{Br}}

\newcommand{\CH}{\operatorname{CH}}

\newcommand{\cl}{\operatorname{cl}}

\newcommand{\Char}{\operatorname{char}}

\newcommand{\et}{\text{\'et}}
\newcommand{\zar}{\text{Zar}}
\newcommand{\proet}{\text{pro\'et}}

\newcommand{\colim}{\operatorname{colim}}

\newcommand{\an}{\operatorname{an}}

\newcommand{\DM}{\operatorname{DM}}
\newcommand{\h}{\operatorname{h}}

\newcommand{\dashedlongrightarrow}{\xymatrix@1@=15pt{\ar@{-->}[r]&}}
\renewcommand{\longrightarrow}{\xymatrix@1@=15pt{\ar[r]&}}
\renewcommand{\mapsto}{\xymatrix@1@=15pt{\ar@{|->}[r]&}}
\renewcommand{\twoheadrightarrow}{\xymatrix@1@=15pt{\ar@{->>}[r]&}}
\newcommand{\hooklongrightarrow}{\xymatrix@1@=15pt{\ar@{^(->}[r]&}}
\newcommand{\congpf}{\xymatrix@1@=15pt{\ar[r]^-\sim&}}
\renewcommand{\cong}{\simeq}

\begin{document}

\title[Truncated pushforwards and refined unramified cohomology]{Truncated pushforwards and refined unramified cohomology}

\author{Theodosis Alexandrou} 
\address{Institute of Algebraic Geometry, Leibniz University Hannover, Welfengarten 1, 30167 Hannover , Germany.}
\email{alexandrou@math.uni-hannover.de} 

\author{Stefan Schreieder} 
\address{Institute of Algebraic Geometry, Leibniz University Hannover, Welfengarten 1, 30167 Hannover , Germany.}
\email{schreieder@math.uni-hannover.de}

\date{\today} 
%


 \begin{abstract}  
For a large class of cohomology theories, we prove that refined unramified cohomology is canonically isomorphic to the hypercohomology of a natural truncated complex of Zariski sheaves. 
This generalizes a classical result of Bloch and Ogus and solves a conjecture of Kok and Zhou. 
 \end{abstract}

\maketitle

\section{Introduction}

Let $X$ be a smooth variety over a field $k$.
We consider the natural map $\pi:X_{\et}\to X_{\zar}$ between the \'etale site and the Zariski site of $X$.
Let $m$ be an integer invertible in $k$.
Then the total pushforward $\R \pi_\ast  \mu_m^{\otimes n}\in D^b(X_\zar)$ is known to capture the motivic cohomology of $X$ with values in $\Z/m$.
Indeed, as a consequence of the Bloch--Kato conjecture, proven by Voevodsky \cite{Voevodsky}, we have 
\begin{align}\label{eq:H_motivic}
H^i_M(X,\Z/m(n))\cong H^i(X,\tau_{\leq n} \R \pi_\ast  \mu_m^{\otimes n}),
\end{align}
see \cite[Corollary 1.2]{geisser-levine}.
Since $H^i_M(X,\Z/m(n))\cong \CH^{n}(X,2n-i,\Z/m)$ agrees with Bloch's higher Chow groups (with finite coefficients), it also follows that the Zariski cohomology of the truncated complex $\tau_{\leq n} \R \pi_\ast  \mu_m^{\otimes n}$ has an explicit geometric meaning in terms of algebraic cycles on $X\times \Delta ^q$.
Similar results hold for $m=p^r$, when $k$ is a perfect field of characteristic $p$, see \cite{geisser-levine-inventiones}.
If $k$ contains a primitive $m$-th root of unity, then $\mu_m^{\otimes n}$ can be replaced by $\Z/m$ in this discussion.

It is natural to wonder whether the hypercohomology of ``the other'' truncation $\tau_{\geq s} \R\pi_\ast \mu_m^{\otimes n}$ admits a natural geometric description as well.
In a recent paper, Kok and Zhou \cite{kok-zhou} found the following conjectural answer to this question.

\begin{conjecture}[{\cite[Conjecture 1.4]{kok-zhou}}] \label{conj:kok-zhou}
Let $X$ be a smooth variety over a field $k$ and let $A(n)$ be a (twisted) locally constant torsion \'etale sheaf in which the exponential characteristic of $k$ is invertible.
Then there are canonical isomorphisms
$$
H^i(X,\tau_{\geq s}\R \pi_\ast A(n) )\cong H^i_{i-s,nr}(X,A(n)) ,
$$
where the right hand side denotes the refined unramified cohomology groups from \cite{Sch-refined} and
where $\pi:X_{\et}\to X_{\zar}$ denotes the canonical map.
Moreover, if $k=\C$ and $X_{\et}$ is replaced by the analytic site $X_{\an}$, then $A(n)$ can be replaced by any abelian group.
\end{conjecture}

We recall that the refined unramified cohomology groups are defined by
$$
H^i_{j,nr}(X,A(n))=\im ( H^i(F_{j+1}X,A(n))\to H^i(F_jX,A(n))), 
$$
where $F_jX=\{x\in X\mid \codim (x)\leq j\}$ and $H^i(F_jX,A(n))=\colim_{U\supset F_jX} H^i(U,A(n))$. 
Equivalently, $H^i_{j,nr}(X,A(n))$ is the subgroup of $H^i(F_jX,A(n))$, given by those classes that have trivial residues at all codimension $j+1$ points, see \cite[Lemma 5.8]{Sch-refined}.
In particular, for $j=0$, this definition naturally recovers classical unramified cohomology,  see \cite[Theorem 4.1.1(a)]{CT} and \cite{CTO}. 
Other special cases are ordinary cohomology (for $j\geq \dim X$) or Kato homology, see \cite[\S 1.3]{Sch-refined}.

The case $i=s$ of Conjecture \ref{conj:kok-zhou} implies by the hypercohomology spectral sequence the isomorphism
$$
H^i_{0,nr}(X,A(n))\cong H^0(X_{\zar},\R^i\pi_\ast A(n)).
$$
This is a celebrated result of  Bloch and Ogus \cite{BO}, which identifies the unramified cohomology of a smooth variety by the global sections of a certain Zariski sheaf.
Conjecture \ref{conj:kok-zhou} may be seen as a generalization of this result.
This has, as we shall discuss below, various interesting consequences.

The purpose of this paper is to prove (a generalization of) Conjecture \ref{conj:kok-zhou}.  
To state our result, consider the following examples of a Grothendieck topology $\nu$ on a $k$-scheme $X$ and a complex $K^\bullet\in D(X_\nu,\Z)$ of sheaves of abelian groups on $X_\nu$: 
\begin{enumerate}
\item \label{item:1:etale-coho} $k$ any field, $\nu=\et$ the small \'etale site of $X$, and $K^\bullet$ the pullback of a bounded below complex of (arbitrary) \'etale sheaves on $\Spec k$. 
\item  \label{item:2:Betti-coho} $k=\C$, $\nu=\an$ the analytic site on $X(\C)$ and $K^\bullet$ any bounded below complex of constant sheaves of abelian groups. 
\item \label{item:4:algebraic-de-Rham} $k$ any field of characteristic zero, $\nu=\zar$ the Zariski site on $X $ and $K^\bullet=\Omega^\bullet_{X/k}$ the algebraic de Rham complex of $X$ over $k$.
\item  \label{item:3:log-de-Rham-Witt} $k$ any perfect field of characteristic $p>0$, $\nu=\et$ the small \'etale site and $K^\bullet$ (some shift of) the logarithmic de Rham Witt sheaf $W_r\Omega_{X,\log}^n$, see \cite{illusie}. 
\item \label{item:7:de-Rham-Witt} $k$ any perfect field of characteristic $p>0$, $\nu=\zar$ the small Zariski site and $K^{\bullet}$ the de Rham Witt complex $W\Omega^{\bullet}_{X}:=\varprojlim W_{n}\Omega^{\bullet}_{X}$ of $X$, see \cite{illusie}. 
\item \label{item:5:proetal} $k$ any field, $\nu=\proet$ the small pro-\'etale site of Bhatt--Scholze and $K^\bullet$ the pullback of a constructible complex in $D_{\operatorname{cons}}((\Spec k)_{\proet},\widehat \Z_\ell)$, see \cite{BS}.
\item \label{item:6:motivic-cohomology} $k$ a perfect field, $\nu=\zar$ the Zariski site  and 
$$
K^\bullet:=A_{X}(n)_{\zar}:=z^{n}(-_{\zar},\bullet)[-2n]\otimes^{\mathbb{L}}A
$$ 
Bloch's cycle complex with values in an abelian group $A$, see \cite{bloch-motivic}. 
\item \label{item:7:etale-motivic-cohomology} $k$ any field, $\nu=\et$ the \'etale site, $A$ an abelian group in which the characteristic exponent $p$ of $k$ is invertible and 
$$
K^\bullet:=A_{X}(n)_{\et}:=z^{n}(-_{\et},\bullet)[-2n]\otimes^{\mathbb{L}}A
$$
the \'etale sheafification of Bloch's cycle complex with values in $A$.
\end{enumerate}

\begin{theorem} \label{thm:main-intro}
Let $X$ be a smooth equi-dimensional algebraic $k$-scheme.
Let $\nu$ be a Grothendieck topology on $X$ that contains all Zariski open covers and let $K^\bullet\in D(X_\nu,\Z)$ be a complex of abelian sheaves as in examples \eqref{item:1:etale-coho}--\eqref{item:7:etale-motivic-cohomology} above. 
Let $\pi:X_\nu\to X_{\zar}$ be the natural morphism of sites. 
Then for all integers $i,s,$ there is a canonical isomorphism: 
\begin{equation*} 
\HH^{i}(X_{\zar},\tau_{\geq s}\R\pi_{\ast}K^{\bullet})\cong \HH^{i}_{i-s,nr}(X_{\nu},K^{\bullet}),
\end{equation*} 
where  $\HH^{i}_{j,nr}(X_{\nu},K^{\bullet}):=\im(\HH^{i}(F_{j+1}X,K^{\bullet})\to \HH^{i}(F_{j}X,K^{\bullet}))$ and $\HH^{i}(F_{j}X,K^{\bullet})=\colim_{U\supset F_jX} H^i(U,K^{\bullet})$. 
\end{theorem}
 
Theorem \ref{thm:main-intro} applied to (special cases of) examples \eqref{item:1:etale-coho} and \eqref{item:2:Betti-coho} proves Conjecture \ref{conj:kok-zhou}.

For a large class of site theoretic cohomology theories, the theorem yields an explicit geometric interpretation of the Zariski cohomology $H^i(X,\tau_{\geq s}\R \pi_\ast K^\bullet )$ of the truncated pushforward $\tau_{\geq s}\R \pi_\ast K^\bullet $ via (unramified) cohomology classes on some open subsets of $X$, thereby complementing the aforementioned description of $H^i(X,\tau_{\leq n}\R \pi_\ast \mu_m^{\otimes n} )$ in terms of higher Chow groups with finite coefficients.  
%
In particular, we get the following, where 
 $$ 
 H^i_M(X,\Z/m(n)):=H^i(X_{\zar},(\Z/m)_{X}(n)_{\zar})\ \ \ \text{and}\ \ \ H^i_L(X,\Z/m(n)):=H^i(X_{\et},(\Z/m)_{X}(n)_{\et})
 $$ 
 denote  motivic and \'etale motivic (or Lichtenbaum) cohomology, respectively. 

  \begin{corollary} \label{cor:motivic}
Let $X$ be a smooth equi-dimensional algebraic scheme over a perfect field $k$ and let $m$ be an arbitrary positive integer.
Then there is a natural long exact sequence
$$
\cdots \to H^i_M(X,\Z/m(n)) \stackrel{\cl}\to H^i_L(X,\Z/m(n))\to H^{i}_{i-n-1,nr}(X_{\et}, (\Z/m)_{X}(n)_{\et})\to H^{i+1}_M(X,\Z/m(n)) \stackrel{\cl}\to \cdots .
$$
\end{corollary}

 By work of Geisser--Levine \cite{geisser-levine-inventiones,geisser-levine}, we have canonical quasi-isomorphisms
\begin{align*} 
(\Z/m)_{X}(n)_{\et}\cong \begin{cases}
\mu_m^{\otimes n}\ \ \ \ &\text{if $m$ is coprime to $\operatorname{char}(k)$;}\\
 W_r\Omega_{X,\log}^n[-n]   \ \ \ \ &\text{if $m=p^r$ and $p=\operatorname{char}(k)>0$.}
\end{cases} 
\end{align*}
Hence the groups $H^i_L(X,\Z/m(n))$ identify to \'etale cohomology and logarithmic de Rham Witt cohomology, respectively, and the change of topology map $\cl$ may be identified with a natural cycle class map for motivic cohomology (resp.\ higher Chow groups) with finite coefficients.
Corollary \ref{cor:motivic} shows that the kernel and cokernel of these cycle class maps are controlled by refined unramified cohomology.
The result was proven via different arguments for quasi-projective varieties and with
$m$ coprime to the characteristic by Kok and Zhou \cite{kok-zhou} (which partly motivated their Conjecture \ref{conj:kok-zhou}). 

The following consequence of Theorem \ref{thm:main-intro} can be regarded as an integral version of Corollary \ref{cor:motivic}. 

\begin{corollary}\label{cor:motivic-integral} 
Let $X$ be a smooth and equi-dimensional algebraic scheme over a perfect field $k$. Then there is a natural long exact sequence:\begin{equation}\label{eq:les-integral}
    \cdots\longrightarrow H^{j}_{M}(X,\Z(i))\longrightarrow H^{j}_{L}(X,\Z(i))\longrightarrow H^{j-1}_{j-i-2,nr}(X_{\et},(\Q/\Z)_{X}(i)_{\et})\longrightarrow  H^{j+1}_{M}(X,\Z(i))\longrightarrow \cdots,
\end{equation}
where $(\Q/\Z)_{X}(i)_{\et}$ is Bloch's cycle complex \cite{bloch-motivic} on $X_{\et}$ with values in $\Q/\Z$.
\end{corollary}

We refer to Corollaries \ref{cor:simple-version-motivic-integral} and \ref{cor:consequence-integral-les} below where direct consequences of Corollary \ref{cor:motivic-integral} are discussed. 
The special cases of the above corollary that concern classical unramified cohomology in degree 3 and 4 have previously been proven in \cite[Theorem 1.1]{kahn96} and \cite[Remarques 2.10 (2)]{kahn}, respectively (see also \cite[Proposition 2.9]{kahn}, \cite[(1.5)]{CT-kahn}, and \cite[p.\ 531, Remark 5.2 (b)]{RS-JIMJ}).

Away from the characteristic, refined unramified cohomology was previously known to be closely related to algebraic cycles \cite{Sch-refined}, generalizing various previous results for ordinary unramified cohomology and cycles of low (co-)dimensions from \cite{CTV,kahn,Voi-unramified,Ma}.
Comparison results between algebraic cycles and
 refined unramified cohomology, together with the fact that the latter can be represented by explicit cohomology classes on certain open subsets of $X$, have already seen various applications, for instance to prove nontriviality of torsion classes in Griffiths groups \cite{Sch-griffiths,Ale-griffiths}, to prove non-algebraicity of certain Tate and Hodge classes \cite{kok}, and to study zero-cycles over non-closed fields \cite{alexandrou-schreieder,Ale-zero-cycles}. 
However, the previous techniques did not allow to tackle the case of $p$-torsion coefficients in characteristic $p$, partly because the strong form of purity that is used in all of the above works fails for logarithmic de Rham Witt cohomology, see e.g.\ \cite[p.\ 45, Remarque]{gros}.

In the case of logarithmic de Rham Witt cohomology, Theorem \ref{thm:main-intro} yields the following explicit calculations for the refined unramified cohomology groups.

\begin{corollary}\label{cor:log-de-Rham-Witt} 
Let $k$ be a perfect field of positive characteristic $p>0$. Let $X$ be a smooth and equi-dimensional algebraic $k$-scheme. Then for all integers $i,j$, we have canonical isomorphisms 
\begin{equation}\label{eq:log-de-Rham-Witt} H^{i}_{j,nr}(X_{\et},(\Z/p^{r})_{X}(n))\cong\begin{cases} H^{i}_{L}(X,\Z/p^{r}(n)),\ &\text{if}\ \ j\geq i-n \\
    H^{j}(X_{\zar}, \R^{i-j}\pi_{\ast}(\Z/p^{r})_{X}(n)),\ &\text{if}\ \ j=i-n-1 \\
    0,\ &\text{if}\ \  j \leq i-n-2,
\end{cases}    
\end{equation}
where the complex $(\Z/p^{r})_{X}(n)$ denotes 
the $-n$-shifted logarithmic de Rham Witt sheaf $W_{r}\Omega^{n}_{X,\log}$ from \cite{illusie} and where $\pi: X_{\et}\to X_{\zar}$ is the natural map of sites induced by the identity.
\end{corollary}

Another direct application of Theorem \ref{thm:main-intro} is as follows.

\begin{corollary} \label{cor:functoriality}
Let $f:X\to Y$ be a morphism between equi-dimensional smooth algebraic $k$-schemes and let $k$, $\nu$ and $K_Y^\bullet\in D(Y_\nu,\Z)$ be as in one of the examples \eqref{item:1:etale-coho}--\eqref{item:7:etale-motivic-cohomology} above.  
Let $K_X^\bullet\in D(X_\nu,\Z)$ be the analogous complex on $X_\nu$.
Then for all $i,j$ there is a functorial pullback map
$$
f^\ast:H^i_{j,nr}(Y,K_Y^\bullet)\longrightarrow H^i_{j,nr}(X, K_X^\bullet),
$$  
which is compatible with the canonical pullback $H^i(Y,K_Y^\bullet)\to H^i(X,K_X^\bullet)$ in the sense that the following diagram commutes: 
\begin{equation*}
    \begin{tikzcd}
{H^i(Y,K_Y^\bullet)} \arrow[r, "f^{\ast}"] \arrow[d] & { H^i(X,K_X^\bullet)} \arrow[d] \\
{H^i_{j,nr}(Y,K_Y^\bullet)} \arrow[r, "f^{\ast}"]    & {H^i_{j,nr}(X, K_X^\bullet).}   
\end{tikzcd}
\end{equation*}
\end{corollary}

The existence of pullbacks for refined unramified cohomology is a priori not clear.
Indeed, a class in $H^i_{j,nr}(Y,K_Y^\bullet)$ is represented by an unramified cohomology class on an open subset $V$ of $Y$ whose complement has codimension $j+1$ in $Y$.
This class can be pulled back to a class on 
$U:=f^{-1}(V)$, but unless $f$ is flat, it is not clear that the complement of $U$ in $X$ has codimension $j+1$.
Hence there is no direct way to define pullbacks.
This issue has been solved in \cite{Sch-moving} for various cohomology theories (including $\ell$-adic (pro-)\'etale cohomology) of quasi-projective varieties in characteristic zero via a moving lemma for cohomology with support and in \cite{kok-zhou2} for $\ell$-adic \'etale cohomology of any smooth variety via deformation to the normal cone.
The above result is more general and covers in particular the case of $p$-torsion coefficients in characteristic $p$, which is new.
In fact, the existence of functorial pullbacks for refined unramified cohomology in the case of non homotopy invariant examples e.g.\ the example \eqref{item:7:de-Rham-Witt} which computes crystalline cohomology, does not follow from the methods considered in \cite{kok-zhou2}, see \cite[Definition 2.4]{kok-zhou2}.\par

In most of the examples \eqref{item:1:etale-coho}--\eqref{item:7:etale-motivic-cohomology}, it is easy to see the existence of pushforwards and exterior products for refined unramified cohomology (via its very definition, not via Theorem \ref{thm:main-intro}).
Combining this with the pullbacks from Corollary \ref{cor:functoriality} one can reprove the results in \cite{Sch-moving,kok-zhou2}.
Our argument is slightly more general and allows us to treat for instance refined unramified cohomology $H^{i}_{j,nr}(X_{\zar},W\Omega^{\bullet}_{X})$ associated to the complex in \eqref{item:7:de-Rham-Witt}, which computes integral crystalline cohomology (see \cite[p.\ 606, Th\'eor\`eme II.1.4]{illusie}) and is not covered in \cite{Sch-moving,kok-zhou2}.  

\begin{corollary}\label{cor:action-correspondences} Let $k$ be a perfect field of characteristic $p>0$. Let $X$ and $Y$ be smooth and equi-dimensional algebraic schemes over $k$. Assume further that $X$ is proper over $k$. If we set $d_{X}:=\dim X$, then for all $c,i,j\geq 0,$ there is a bi-additive pairing 
\begin{equation}\label{eq:action-correspondences}
    \CH^{c}(X\times Y)\times H^{i}_{j,nr}(X,W\Omega^{\bullet}_{X})\longrightarrow H^{i+2c-2d_{X}}_{j+c-d_{X},nr}(Y,W\Omega^{\bullet}_{Y}), \ \ \  ([\Gamma],\alpha)\mapsto [\Gamma]_{\ast}(\alpha),
\end{equation} 
which is functorial with respect to the composition of correspondences.
\end{corollary}

Theorem \ref{thm:main-intro} will be deduced from the following more general result, which applies essentially to any site theoretic cohomology theory that satisfies a version of the Gersten conjecture.

\begin{theorem}\label{thm:main} Let $X$ be an equi-dimensional algebraic $k$-scheme. 
Let $\pi: X_{\nu}\to X_{\zar}$ be a morphism of sites associated to some Grothendieck topology $\nu$ on $(\text{Sch}/k)$ that contains all Zariski open coverings.  
Suppose that there is a complex $\Co^{\bullet}\in D(X_{\nu},\Z)$ such that for all $i$, there is a resolution 
$\epsilon^{i}:\R^{i}\pi_{\ast}\Co^{\bullet}\to \mathcal{E}^{i,\bullet}_{X}$ in $ D^{+}(X_{\zar},\Z)$ by a complex $\mathcal{E}^{i,\bullet}_{X}$ concentrated in non-negative degrees of the form: 
\begin{equation}\label{eq:flasque-resolution}
\mathcal{E}^{i,\bullet}_{X}:\ 
0\longrightarrow \bigoplus_{x\in X^{(0)}} \iota_{x\ast}\underline {\A}^{i}_{x}\longrightarrow \bigoplus_{x\in X^{(1)}} \iota_{x\ast}\underline {\A}^{i+1}_{x} \longrightarrow \bigoplus_{x\in X^{(2)}} \iota_{x\ast}\underline {\A}^{i+2}_{x} \longrightarrow \cdots \longrightarrow \bigoplus_{x\in X^{(r)}} \iota_{x\ast}\underline {\A}^{i+r}_{x}\longrightarrow \cdots,
\end{equation} 
where $\iota_{x\ast}\underline {\A}^{i+j}_{x}$ placed in degree $j$
is a constant sheaf supported on $\overline{\{x\}}\subset X$ which corresponds to some abelian group $\A^{i+j}_{x}$.

\par Then for all integers $i,s,$ there is a canonical isomorphism: 
\begin{equation}\label{eq:can-iso}\HH^{i}(X_{\zar},\tau_{\geq s}\R\pi_{\ast}\Co^{\bullet})\cong \HH^{i}_{i-s,nr}(X_{\nu},\Co^{\bullet}),
\end{equation} 
where $\HH^{i}_{j,nr}(X_{\nu},\Co^{\bullet}):=\im(\HH^{i}(F_{j+1}X,\Co^{\bullet})\to \HH^{i}(F_{j}X,\Co^{\bullet}))$ and $\HH^{i}(F_{j}X,\Co^{\bullet})=\colim_{U\supset F_jX} H^i(U,\Co^{\bullet})$. 
\end{theorem}
 
Note that the resolution in \eqref{eq:flasque-resolution} is automatically acyclic, because the sheaves $\iota_{x\ast}\underline {\A}^{i+j}_{x}$ are flasque on the Zariski site of $X$.

 \par 
In practice, the resolution \eqref{eq:flasque-resolution} that we consider is  given by instances where Gersten's conjecture holds, see e.g.\ \cite{quillen,BO,gros-suwa,CTHK,Sch-moving}. 
This requires $X$ to be smooth, even though the formal set-up in the above theorem does not need this requirement. 

\section{Notation, conventions, and preliminaries.}
The exponential characteristic of a field $k$ of characteristic $p\geq 0$ is $1$ if $p=0$ and it is $p$ otherwise. 

An algebraic scheme is a separated scheme of finite type over a field.
A variety is an integral algebraic scheme.
 
If $X$ is a scheme and $\nu$ denotes a Grothendieck topology on $X$, then $D(X_{\nu},\Z)$ denotes the derived category of the abelian category of sheaves of abelian groups on $X_\nu$. 
We further denote by $D^+(X_\nu,\Z)$ the full triangulated subcategory that consists of objects that have trivial cohomology sheaves in sufficiently negative degrees. 
For a complex $K^\bullet\in D(X_\nu,\Z)$, we denote by $\mathcal H^n(K^\bullet)$ its $n$-th cohomology sheaf.

Recall that for any complex $K^{\bullet}\in D (X_{\nu},\Z)$ there is a K-injective replacement $K^\bullet\to I^\bullet$, see
 \cite[\href{https://stacks.math.columbia.edu/tag/079P}{Tag 079P}]{stacks-project}. 
 Using this, any left exact functor can be derived, see  \cite[\href{https://stacks.math.columbia.edu/tag/079V}{Tag 079V}]{stacks-project}. 
For a complex $K^\bullet\in D(X_\nu,\Z)$ and a closed subset $Z\subset X$, we use the notations 
$$
H^i(X_\nu,K^\bullet):=R^i\Gamma(X_\nu,K^\bullet)\ \ \text{and}\ \ H^i_Z(X_\nu,K^\bullet):=R^i\Gamma_Z(X_\nu,K^\bullet) ,
$$
where $\Gamma$ and $\Gamma_Z$ denote the global section functor, and the global section functor with support, respectively. 
If $K^\bullet=\mathcal F[0]$ is a sheaf placed in degree zero, then we write $H^i(X_\nu,\mathcal F):=H^i(X_\nu, \mathcal F[0])$. 
If $j:U\hookrightarrow X$ is an open immersion, then we write
$$
H^i(U_\nu, K^\bullet):=H^i(U_\nu,j^\ast K^\bullet).
$$

We write $F_{j}X:=\{x\in X\ |\ \codim(x)\leq j\},$ where $\codim(x):=\dim X-\dim \overline{\{x\}}$.
This may be seen as a pro-scheme which consists of all open subsets of $X$ that contain all codimension $j$ points of $X$.
For any Grothendieck topology $\nu$ on $(\text{Sch}/k)$ and any complex $K^{\bullet}\in D (X_{\nu},\Z)$, we define the group 
\begin{equation}\label{eq:def-F_{j}X} \HH^{i}(F_{j}X,K^{\bullet}):=\varinjlim \HH^{i}(U_{\nu},K^{\bullet}),\end{equation}
 where the direct limit is taken over all Zariski open subsets $U\subset X$ with $F_{j}X\subset U$.
If $U\subset V\subset X$ are Zariski open subsets, then there are canonical restriction maps $H^i(V_\nu,K^\bullet)\to H^i(U_\nu,K^\bullet)$.
These maps induce natural restriction maps
$$
 \HH^{i}(F_{j+1}X,K^{\bullet})\longrightarrow  \HH^{i}(F_{j}X,K^{\bullet}).
$$
In analogy to \cite[Definition 5.1]{Sch-refined}, we define the image of this map as the $j$-th refined unramified cohomology of the complex $K^\bullet\in D(X_\nu,\Z)$ by
\begin{equation}\label{eq:refined-unramified}  
H^i_{j,nr}(X,K^\bullet):=\im\left( \HH^{i}(F_{j+1}X,K^{\bullet})\to  \HH^{i}(F_{j}X,K^{\bullet})  \right).
\end{equation}

\section{Proof of Theorem \ref{thm:main}}
The following is Theorem \ref{thm:main} with a slightly different but equivalent formulation.
\begin{theorem}\label{thm:main-general} Let $X$ be an equi-dimensional algebraic $k$-scheme. Let $K^{\bullet}\in D(X_{\zar},\Z)$ be a complex, such that for all $i$, there is a resolution $\epsilon^{i}:\mathcal{H}^{i}(K^{\bullet})\to \mathcal{E}^{i,\bullet}_{X}$ in $D^{+}(X_{\zar},\Z)$ of the $i$-th Zariski cohomology sheaf of $K^{\bullet}$ by a complex $\mathcal{E}^{i,\bullet}_{X}$ concentrated in non-negative degrees of the form: 
\begin{equation}\label{eq:flasque-resolution-general}
\mathcal{E}^{i,\bullet}_{X}:\ 
0\longrightarrow \bigoplus_{x\in X^{(0)}} \iota_{x\ast}\underline {\A}^{i}_{x}\longrightarrow \bigoplus_{x\in X^{(1)}} \iota_{x\ast}\underline {\A}^{i+1}_{x} \longrightarrow \bigoplus_{x\in X^{(2)}} \iota_{x\ast}\underline {\A}^{i+2}_{x} \longrightarrow \cdots \longrightarrow \bigoplus_{x\in X^{(r)}} \iota_{x\ast}\underline {\A}^{i+r}_{x}\longrightarrow \cdots,
\end{equation} 
where $\iota_{x\ast}\underline {\A}^{i+j}_{x}$ placed in degree $j$ is a constant sheaf supported on $\overline{\{x\}}\subset X$ which corresponds to some abelian group $\A^{i+j}_{x}$.

\par Then for all integers $i,s,$ there is a canonical isomorphism: 
\begin{equation}\label{eq:can-iso-general}\HH^{i}(X_{\zar},\tau_{\geq s} K^{\bullet})\cong \HH^{i}_{i-s,nr}(X_{\zar},K^{\bullet}),
\end{equation} 
where the group $\HH^{i}_{j,nr}(X_{\zar},K^{\bullet})$ is the $j$-th refined unramified cohomology defined in \eqref{eq:refined-unramified}.
\end{theorem}

\begin{proof} We prove a series of preparatory lemmas that combined all together yield Theorem \ref{thm:main-general}. We begin with the following: 

\begin{lemma}\label{lem:restriction-maps-F_{j}X} The natural restriction map \begin{equation*} \HH^{i}(F_{j+1}X,\tau_{\geq s}K^{\bullet})\longrightarrow \HH^{i}(F_{j}X,\tau_{\geq s}K^{\bullet})
\end{equation*} is an isomorphism if $j>i-s$ and injective for $j=i-s$.
\end{lemma}

\begin{proof} For any closed subscheme $\iota: Z\hookrightarrow X$ of pure codimension $c>0$ with complement $j_{U}:U\hookrightarrow X$, we consider the complex $\iota^{!}\mathcal{E}^{q,\bullet}_{X}\in D^{+}(Z_{\zar},\Z)$ defined by 
\begin{equation}\label{def:i^{!}E}
\iota_{\ast}\iota^{!}\mathcal{E}^{q,\bullet}_{X}:=\ker(\alpha_{j_{U}}:\mathcal{E}^{q,\bullet}_{X}\longrightarrow j_{U\ast}j^{\ast}_{U}\mathcal{E}^{q,\bullet}_{X}),
\end{equation} 
where the above kernel is taken in the abelian category of cochain complexes of abelian sheaves on $X_{\zar}$ and the map $\alpha_{j_{U}}$ is the unit of the relevant adjunction.\par

By evaluating \eqref{def:i^{!}E} at open subsets of $X$, the set-theoretic equality $X^{(p)}\setminus U^{(p)}=Z^{(p-c)}$ yields that 
$$
\iota^{!}\mathcal{E}^{q,p}_{X}=\bigoplus_{x\in Z^{(p-c)}}\iota_{x\ast}\underline{\A}^{q+p}_{x}
$$ 
is a complex of flasque sheaves on the Zariski site.
Moreover, we find that $\iota^{!}\mathcal{E}^{q,p}_{X}=0$ in all degrees $p<c$. 
Hence $\iota^{!}\mathcal{E}^{q,\bullet}_{X}=\mathcal{E}^{q+c,\bullet}_{Z}[-c]$, where $\mathcal{E}^{q+c,\bullet}_{Z}\in D^{+}(Z_{\zar},\Z)$ with
$$
\mathcal{E}^{q+c,p}_{Z}=\bigoplus_{x\in  Z^{(p)}} \iota_{x\ast} \underline{\A}_x^{q+c+p}
$$
is a complex of the form \eqref{eq:flasque-resolution-general} on $Z$ (by no means we claim here that the latter complex is exact in degrees $>0$).
 It follows directly from the definition of the derived functor for the global section functor with support $\Gamma_{Z}$ that for any non-empty open subset $V\subset X$, we get natural isomorphisms \begin{equation}\label{eq:can-iso-derived}
    \R\Gamma_{Z\cap V}(V_{\zar},\mathcal{H}^{q}(K^{\bullet}))\cong \Gamma(Z\cap V,\iota^{!}\mathcal{E}^{q,\bullet}_{X})= \Gamma(Z\cap V,\mathcal{E}^{q+c,\bullet}_{Z})[-c],
\end{equation}
that are contravariantly functorial with respect to $V$ and covariantly functorial with respect to equi-dimensional closed subsets $Z\subset Z'\subset X$ of the same codimension $c$ in $X$. Note that above we explicitly use that $\mathcal{H}^{q}(K^{\bullet})\to \mathcal{E}^{q,\bullet}_{X}$ is an acyclic (flasque) resolution for the global section functor with support. In particular, taking cohomology of \eqref{eq:can-iso-derived} gives: \begin{equation}\label{eq:can-iso-coh} H^{p}_{Z\cap V}(V_{\zar},\mathcal{H}^{q}(K^{\bullet}))\cong H^{p-c}(\Gamma(Z\cap V,\mathcal{E}^{q+c,\bullet}_{Z})).\end{equation} 

\par Now pick a pair $(Z,W)$ of equi-dimensional closed subsets $W\subset Z\subset X$ such that $\dim Z=\dim X-(j+1)$ and $\dim W=\dim X-(j+2)$. We assign to $(Z,W)$ a natural long exact localisation sequence: 
\begin{equation}\label{eq:long-exact-seq-(Z,W)}\cdots\longrightarrow H^{i}_{Z\setminus W}(X\setminus W,\tau_{\geq s}K^{\bullet})\longrightarrow H^{i}(X\setminus W,\tau_{\geq s}K^{\bullet})\longrightarrow H^{i}(X\setminus Z,\tau_{\geq s}K^{\bullet})\longrightarrow \cdots,\end{equation} which we denote by $\textbf{L-SEQ}_{(Z,W)}$.

\par Let $\mathcal{I}$ denote the index set, whose elements are pairs $(Z,W)$ of equi-dimensional closed subsets $W\subset Z\subset X$ as above, i.e. $\dim Z=\dim X-(j+1)$ and $\dim W=\dim X-(j+2)$. We turn $\mathcal{I}$ into a directed set by declaring $(Z',W')\leq (Z,W)$ if and only if $Z\subset Z'$ and $W\subset W'$. It is readily seen that for $(Z',W')\leq (Z,W)$, we get a canonical map $\textbf{L-SEQ}_{(Z,W)}\to \textbf{L-SEQ}_{(Z',W')}$. Indeed, this is achieved by first restricting the sequence $\textbf{L-SEQ}_{(Z,W)}$ to the open subset $X\setminus W'$ and then composing with \begin{equation*}
    \begin{tikzcd}
{\cdots} \arrow[r] & {H^{i}_{Z\setminus W'}(X\setminus W')} \arrow[r] \arrow[d] & {H^{i}(X\setminus W')} \arrow[r] \arrow[d, equal] & {H^{i}(X\setminus (W'\cup Z))} \arrow[r] \arrow[d] & {\cdots} \\
{\cdots} \arrow[r] & {H^{i}_{Z'\setminus W'}(X\setminus W')} \arrow[r]          & {H^{i}(X\setminus W')} \arrow[r]                                & {H^{i}(X\setminus Z')} \arrow[r]                   & {\cdots} .
\end{tikzcd}
\end{equation*} 
In particular, this gives a direct system over $\mathcal{I}$. Taking the direct limit over this index set, \eqref{eq:long-exact-seq-(Z,W)} yields \begin{equation}\label{eq:loc-coh-seq-F_{j}}\cdots\longrightarrow \bigoplus_{x\in X^{(j+1)}}H^{i}_{x}(X,\tau_{\geq s}K^{\bullet})\longrightarrow H^{i}(F_{j+1}X,\tau_{\geq s}K^{\bullet})\longrightarrow H^{i}(F_{j}X,\tau_{\geq s}K^{\bullet})\longrightarrow\cdots,\end{equation} where we put $H^{i}_{x}(X,\tau_{\geq s}K^{\bullet}):=\varinjlim H^{i}_{V\cap\overline{\{x\}}}(V,\tau_{\geq s}K^{\bullet})$ and the direct limit here runs through opens $V\subset X$ with $x\in V$.
\par 

The claim thus follows from \eqref{eq:loc-coh-seq-F_{j}} once we show that the groups $H^{i}_{x}(X,\tau_{\geq s}K^{\bullet})$ vanish for all $i\leq j+s$ and all points $x\in X^{(j+1)}$. By exactness of the direct limit functor, we obtain a hypercohomology spectral sequence $$E^{p,q}_{2}:=H^{p}_{x}(X,\mathcal{H}^{q}(\tau_{\geq s}K^{\bullet}))\implies \HH^{p+q}_{x}(X,\tau_{\geq s}K^{\bullet}).$$ We clearly have that $E^{p,q}_{2}=0$ for $q<s$. Thus it is enough to prove that \begin{equation}\label{eq:E2=0}H^{p}_{x}(X,\mathcal{H}^{q}(K^{\bullet}))=0\ \text{for}\ q\geq s,\ p+q=i\ \text{and}\ j\geq i-s.\end{equation} Indeed, let $c:=j+1$ and fix the notation $Z_{x}:=\overline{\{x\}}\subset X$ with $x\in X^{(c)}$. Note that by \eqref{eq:can-iso-coh}, we find a natural isomorphism \begin{equation}\label{eq:can-iso-H_{x}}
H^{p}_{x}(X,\mathcal{H}^{q}(K^{\bullet}))\cong H^{p-c}(\Gamma(F_{0}Z_{x},\mathcal{E}^{q+c,\bullet}_{Z_{x}})),\end{equation} where we set $H^{p}(\Gamma(F_{0}Z_{x},\mathcal{E}^{q+c,\bullet}_{Z_{x}})):=\varinjlim H^{p}(\Gamma(V,\mathcal{E}^{q+c,\bullet}_{Z_{x}}))$ and the direct limit runs through non-empty opens $V\subset Z_{x}$.\par

The inequalities in \eqref{eq:E2=0} guarantee that $p\leq j+s-q<c$. Since the complex $\mathcal{E}^{q+c,\bullet}_{Z_{x}}$ is concentrated in non-negative degrees, the canonical isomorphism \eqref{eq:can-iso-H_{x}} in turn implies the desired vanishing result \eqref{eq:E2=0}. The proof of Lemma \ref{lem:restriction-maps-F_{j}X} is finally complete.\end{proof}

\begin{corollary}\label{cor:restriction-maps-F_{j}X} For all integers $i,s,$ there is a canonical isomorphism \begin{equation*} \HH^{i}(X,\tau_{\geq s}K^{\bullet})\overset{\cong}{\longrightarrow} \HH^{i}_{i-s,nr}(X,\tau_{\geq s}K^{\bullet}).
\end{equation*}
\end{corollary}

\begin{proof} We apply repeatedly Lemma \ref{lem:restriction-maps-F_{j}X} until we reach an isomorphism $$\HH^{i}(F_{j}X,\tau_{\geq s}K^{\bullet})\overset{\cong}{\longrightarrow} \HH^{i}_{i-s,nr}(X,\tau_{\geq s}K^{\bullet})$$ with $j>\max\{\dim X,i-s\}$. The result now follows as $F_{j}X=X$ for $j\geq\dim X$.\end{proof}

\begin{lemma}\label{lem:vanishing-result} We have $\HH^{i}(F_{j}X,\tau_{\leq s}K^{\bullet})=0$ for all $j<i-s$.\end{lemma}

\begin{proof} By exactness of the direct limit functor, we get as before a hypercohomology spectral sequence $$E^{p,q}_{2}:=H^{p}(F_{j}X,\mathcal{H}^{q}(\tau_{\leq s}K^{\bullet}))\implies \HH^{p+q}(F_{j}X,\tau_{\leq s}K^{\bullet}).$$ 
Even though $\tau_{\leq s}K^{\bullet} $ may be unbounded to the left, the above spectral sequence converges as the Zariski topology on $X$ has finite cohomological dimension, see \cite[Th\'eor\`eme 3.6.5]{grothendieck}. It thus suffices to show that \begin{equation*}
    H^{p}(F_{j}X,\mathcal{H}^{q}(K^{\bullet}))=0\ \text{for}\ p+q=i,\ q\leq s\ \text{and}\ j<i-s.\end{equation*}
The two inequalities imply that $q<i-j$ and since $p=i-q$ we find $p>j$. The vanishing in question follows from the elementary observation that the complex $\Gamma(F_{j}X,\mathcal{E}^{q,\bullet}_{X})$ is concentrated in degrees $0\leq p\leq j$. Indeed, recall that the sheaf $\mathcal{E}^{q,p}_{X}$ is a direct sum of sheaves whose supports are closed subvarieties of $X$ of codimension $p$. Since forming the direct limit commutes with the direct sum, we obtain that $\Gamma(F_{j}X,\mathcal{E}^{q,p}_{X})=0$ for $p>j$, as claimed.\end{proof}

\begin{corollary}\label{cor:comparison-result} The natural map \begin{equation*}\HH^{i}(F_{j}X,K^{\bullet})\longrightarrow \HH^{i}(F_{j}X,\tau_{\geq s}K^{\bullet})\end{equation*} is an isomorphism if $j<i-s+1$ and surjective if $j=i-s+1$.\end{corollary}

\begin{proof} Consider the canonical exact triangle $$\tau_{\leq s-1}K^{\bullet}\longrightarrow K^{\bullet}\longrightarrow \tau_{\geq s}K^{\bullet}\overset{+1}{\longrightarrow}$$ of complexes of Zariski sheaves. This gives in turn an exact sequence $$\HH^{i}(F_{j}X,\tau_{\leq s-1}K^{\bullet})\longrightarrow \HH^{i}(F_{j}X,K^{\bullet})\longrightarrow \HH^{i}(F_{j}X,\tau_{\geq s}K^{\bullet})\longrightarrow \HH^{i+1}(F_{j}X,\tau_{\leq s-1}K^{\bullet}),$$ where by Lemma \ref{lem:vanishing-result} the first and the last group vanish when $j<i-s+1$ and $j<i+1-s+1$, respectively. The claim thus follows.\end{proof}

We are now in the position to conclude the proof of Theorem \ref{thm:main-general}. We have a commutative diagram \begin{equation*}\begin{tikzcd}
{\HH^{i}(F_{i-s+1}X,K^{\bullet})} \arrow[r, two heads] \arrow[d] & {\HH^{i}(F_{i-s+1}X,\tau_{\geq s}K^{\bullet})} \arrow[d] \\
{\HH^{i}(F_{i-s}X,K^{\bullet})} \arrow[r, "\cong"]               & {\HH^{i}(F_{i-s}X,\tau_{\geq s}K^{\bullet})},            
\end{tikzcd}\end{equation*} where all maps are the canonical ones and where by Corollary \ref{cor:comparison-result} the upper horizontal map is surjective and the lower one is an isomorphism. This yields an isomorphism \begin{equation*} \HH^{i}_{i-s,nr}(X,K^{\bullet})\cong \HH^{i}_{i-s,nr}(X,\tau_{\geq s}K^{\bullet}). 
\end{equation*}
The canonical isomorphism \eqref{eq:can-iso-general} is then given by the composition $$\HH^{i}(X,\tau_{\geq s}K^{\bullet})\overset{\cong}{\longrightarrow} \HH^{i}_{i-s,nr}(X,\tau_{\geq s}K^{\bullet})\cong \HH^{i}_{i-s,nr}(X,K^{\bullet}),$$ where the first map is an isomorphism by Corollary \ref{cor:restriction-maps-F_{j}X}. This completes the proof of Theorem \ref{thm:main-general}.
\end{proof}

\begin{proof}[Proof of Theorem \ref{thm:main}]
    We replace $\nu$ by the Zariski topology and $K^\bullet$ by the total pushforward $R\pi_\ast K^\bullet$.
    Then Theorem \ref{thm:main-general} applies and gives the result.
\end{proof}

\section{Proof of Theorem \ref{thm:main-intro}}
 \begin{proof}[Proof of Theorem \ref{thm:main-intro}]
Let $k$, $\nu$ and $K^\bullet\in D(X_\nu,\Z)$ be as in one of the examples in \eqref{item:1:etale-coho}--\eqref{item:7:etale-motivic-cohomology}. 
For a closed subset $Z\subset X$, we get cohomology theories with support by
$$
H^i_Z(X):=\R^i\Gamma_Z(X_\nu, K^\bullet),
$$
where $\Gamma_Z$ denotes the global section functor with support. 
If $j:V\hookrightarrow X$ is an open immersion, then we write 
$$
 H^i_{Z\cap V}(V):= \R^i\Gamma_{Z\cap V}(V_\nu, j^\ast K^\bullet) .
$$

If $U\subset X$ denotes the complement of $Z\subset X$, then there is a natural exact triangle $$\R \Gamma_Z(X_{\nu},-)\longrightarrow \R \Gamma (X_{\nu},-)\longrightarrow \R \Gamma(U_{\nu},-)\stackrel{+1}\longrightarrow$$ defined on $D(X_\nu,\Z)$.
In particular, we get residue maps $\del:H^{i}(U)\to H^{i+1}_Z(X)$.
This still works if we replace $X$ by some open subset $V\subset X$ and $K^\bullet$ by its restriction to $V$.
In particular, we get residue maps $\del:H^i(U\cap V)\to H^{i+1}_{Z\cap V}(V)$. 

For a codimension $j$ point $x\in X^{(j)}$ with closure $Z:=\overline{\{x\}}\subset X$, we may now define
$$
H^i_x(X):=\lim_{\substack{\longrightarrow \\ x\in V\subset X}} H^i_{Z\cap V}(V),
$$
where $V\subset X$ runs through all Zariski open subsets of $X$ that contain $x$.
If $V\subset V'\subset X$ are open subsets with $Z'=V'\setminus V$, then the aforementioned residue maps yield a map 
$$
H^i_{Z\cap V}(V)\longrightarrow H^i(V)\stackrel{\del}\longrightarrow H^{i+1}_{Z'}(V').
$$
Taking direct limits, we get maps
$$
\del:\bigoplus_{x\in X^{(j)}} H^i_x(X)\longrightarrow \bigoplus_{x\in X^{(j+1)}} H^{i+1}_x(X).
$$

For a codimension $j$ point $x\in X^{(j)}$, we define
$$
\A^{i+j}_x:=H^{i+j}_x(X)
$$ 
and we let $\iota_{x\ast}\underline{\A}^{i+j}_x$ be the sheaf on $X_{\zar}$ that is constant equal to $\A ^{i+j}_x$ on the closure of $x$ and zero outside of that closed subset.  
The above residue maps thus yield a complex
\begin{equation}\label{eq:flasque-resolution-2}
\mathcal{E}^{i,\bullet}_{X}:\ 
0\longrightarrow \bigoplus_{x\in X^{(0)}} \iota_{x\ast}\underline {\A}^{i}_{x}\longrightarrow \bigoplus_{x\in X^{(1)}} \iota_{x\ast}\underline {\A}^{i+1}_{x} \longrightarrow \bigoplus_{x\in X^{(2)}} \iota_{x\ast}\underline {\A}^{i+2}_{x} \longrightarrow \cdots \longrightarrow \bigoplus_{x\in X^{(r)}} \iota_{x\ast}\underline {\A}^{i+r}_{x}\longrightarrow \cdots,
\end{equation} 
as in \eqref{eq:flasque-resolution}.
There is a natural map of complexes $\R^i\pi_\ast K^\bullet\to \mathcal{E}^{i,\bullet}_{X}$, where we view $\R^i\pi_\ast K^\bullet$ to be concentrated in degree zero.

We need to show that $\R^i\pi_\ast K^\bullet\to \mathcal{E}^{i,\bullet}_{X}$ is a resolution.
It is not hard to see that this follows if one can show a certain effacement theorem, see \cite{quillen,BO} and especially \cite[Proposition 2.1.2 and Theorem 2.2.7]{CTHK}, which asserts that for any open subset $U\subset X$, any finite set $S\subset U$, and any closed subset $Z\subset U$ of pure codimension $j$, there is a closed subset $W\subset U$ with $Z\subset W$ and of pure codimension $j-1$, and a closed subset $Z'\subset W$ with $S\cap Z'=\emptyset$, such that the natural composition
$$ H^i_Z(U)\longrightarrow H^i_W(U)\longrightarrow H^i_{W\setminus Z'}(U\setminus Z') $$
is zero.
Equivalently, the image of $H^i_Z(U)\to H^i_W(U)$ is contained in the image of $H^i_{Z'}(U)\to H^i_W(U)$, where $Z'$ is disjoint from $S$ and hence in ``good'' position with respect to $S$.
In particular, the effacement theorem may be seen as a special case of a moving lemma for cohomology with support, see \cite{Sch-moving}. 

It remains to show that the Gersten conjecture, respectively an ``effacement theorem'', holds for the cohomology theories considered in \eqref{item:1:etale-coho}--\eqref{item:7:etale-motivic-cohomology}. 
This follows in the case of item \eqref{item:1:etale-coho} from \cite[7.4(1)]{CTHK}, in case of items \eqref{item:2:Betti-coho} and \eqref{item:4:algebraic-de-Rham} from \cite[7.3(2)--(3)]{CTHK}, in the cases of logarithmic de Rham Witt and de Rham Witt cohomology as in items \eqref{item:3:log-de-Rham-Witt} and \eqref{item:7:de-Rham-Witt}, it follows from \cite{gros-suwa} and \cite[7.4(3)]{CTHK}, and in the case of $\ell$-adic pro-\'etale cohomology it follows from \cite[Theorem 1.1 and Proposition 3.2(3)]{Sch-moving}.

Next we consider the case of Bloch's cycle complex tensored with an abelian group $A$ as in item \eqref{item:6:motivic-cohomology} and assume that the field $k$ is perfect. We let 
$$
\Z^{SF}(n):=\underline{C}_{\ast}(z_{equi}(\mathbb{A}^{n},0))[-2n]
$$ 
be the motivic complex of weight $n$ as defined by Friedlander and Suslin in \cite[Section 8]{F-S}. Here, $z_{equi}(X,0)$ denotes the \'etale sheaf of equi-dimensional cycles on $X$ of relative dimension $0$ on the category $Sm/k$ of smooth $k$-schemes as given in \cite[Chapter 4, $\S$2, pp. 141]{v-s-f}, whereas the complex $\underline{C}_{\ast}(F)$ assigned to a presheaf $F$ of abelian groups on $Sm/k$ is the singular simplicial complex of $F$, see \cite[Chapter 4, $\S 4$, pp. 150]{v-s-f}. In the notation of item \eqref{item:6:motivic-cohomology}, we have a canonical isomorphism $$\Z^{SF}(n)|_{X_{\zar}}\otimes A\cong A_{X}(n)_{\zar}$$ in the derived category of complexes of sheaves on the small Zariski site $X_{\zar}$, see \cite[Proposition 12.1]{F-S}. In particular, we obtain $$H^{i}_{M}(X,A(n))\cong H^{i}(X_{\zar},\Z^{SF}(n)\otimes A).$$ Note that the above tensor product coincides with the derived one, since the sheaves $\underline{C}_{\ast}(z_{equi}(\mathbb{A}^{n},0))$ are flat. It can be readily seen that $\Z^{SF}(n)\otimes A$ is the Zariski sheafification of the complex $$\underline{C}_{\ast}(z_{equi}(\mathbb{A}^{n},0)\otimes^{PrSh}A)[-2n],$$ where $\otimes^{PrSh}$ denotes the tensor product in the category of presheaves on $Sm/k$. The latter can be viewed as a complex of pretheories (\cite[Chapter 3, Definition 3.1]{v-s-f}) with homotopy invariant cohomology presheaves, see \cite[Chapter 4, $\S5$, Proposition 5.7]{v-s-f} and \cite[Chapter 3, $\S$3, Proposition 3.6]{v-s-f}. It follows from \cite[Chapter 3, $\S$4, Proposition 4.26]{v-s-f} that the Zariski cohomology sheaves of $\Z^{SF}(n)\otimes A$ are homotopy invariant pretheories over $k$ and thus their restriction to $X_{\zar}$ admits a Gersten resolution by \cite[Chapter3, $\S$4, Theorem 4.37]{v-s-f}, as required. 
\par

Lastly, the case of \'etale motivic cohomology as in item \eqref{item:7:etale-motivic-cohomology} (i.e.\ with coefficients in an abelian group $A$ in which the characteristic exponent $p$ is invertible) is essentially contained in \cite{cd}. For the reader's convenience we include some details of the argument. Let $\DM_{\h}(X,A)$ denote the category of $\h$-motives as defined in \cite[Definition 5.1.3]{cd}. The authors in \cite{cd} define \'etale motivic cohomology by $$H^{i}(X,n):=\Hom_{\DM_{\h}(X,A)}(A_{X}, A_{X}(n)[i]),$$ where $A_{X}$ is the identity object for $\otimes$ in $\DM_{\h}(X,A)$ and $A_{X}(1)$ is the Tate object. The above definition agrees with Lichtenbaum cohomology if the characteristic exponent $p$ of $k$ is invertible in $A$ and $X$ is a smooth and equi-dimensional algebraic $k$-scheme, see \cite[Theorem 7.1.2]{cd}. We show that the cohomology theory with supports \begin{equation}\label{eq:et-mot-supp} H^{i}_{Z}(X,n):= \Hom_{\DM_{\h}(Z,A)}(A_{Z}, \iota^{!}A_{X}(n)[i])\end{equation} satisfies the effacement theorem, where $\iota: Z\hookrightarrow X$ is a closed embedding of algebraic $k$-schemes. By \cite[Theorem 5.6.2]{cd}, the triangulated premotivic category $\DM_{\h}(-,A)$ satisfies the formalism of the Grothendieck $6$ functors for Noetherian schemes of finite dimension (\cite[Definition A.1.10]{cd}) along with the absolute purity property (\cite[Definition A.2.9]{cd}). This ensures as in the case of \'etale cohomology with finite coefficients \cite[Example (2.1)]{BO} that the axioms from \cite[Definition (1.1) $\&$ (1.2)]{BO} are satisfied and so the effacement theorem holds, as we want.
 
Altogether, this completes the proof of the theorem. 
\end{proof}
 
\section{Applications}

In this section we prove Corollaries \ref{cor:motivic}, \ref{cor:log-de-Rham-Witt}, \ref{cor:functoriality} and \ref{cor:action-correspondences}.

\subsection{Comparison to the cycle class map on higher Chow groups with finite coefficients}
 
\begin{proof}[Proof of Corollary \ref{cor:motivic}]
 Let $k$ be a perfect field and let $\pi:X_{\et}\to X_{\zar}$ be the natural map of sites from the \'etale site to the Zariski site of $X$ and let $m$ be an arbitrary integer.
Let $(\Z/m)_{X}(n)_{\zar}:=z^{n}(-_{\zar},\bullet)[-2n]\otimes^{\mathbb{L}}\Z/m$ denote Bloch's cycle complex from \cite{bloch-motivic} tensored with $\Z/m$ and let 
 $(\Z/m)_{X}(n)_{\et}$ be its \'etale sheafification.
 
By \cite{geisser-levine-inventiones,geisser-levine}, we have canonical isomorphisms in $D(X_{\et},\Z)$
\begin{align*} 
(\Z/m)_{X}(n)_{\et}\cong \begin{cases}
\mu_m^{\otimes n}\ \ \ \ &\text{if $m$ is coprime to $\operatorname{char}(k)$;}\\
 W_r\Omega_{X,\log}^n[-n]   \ \ \ \ &\text{if $m=p^r$ and $p=\operatorname{char}(k)>0$.}
\end{cases} 
\end{align*}
The higher direct images via $\pi$ of the sheaves/complexes on the right admit Gersten resolutions by \cite{BO,gros-suwa}.
It follows from the Chinese remainder theorem that the same holds for $\R^i\pi_\ast (\Z/m)_{X}(n)_{\et}$ and all $i,m$.
Hence the assumptions of Theorem \ref{thm:main} are satisfied for the complex $(\Z/m)_{X}(n)_{\et}$ and we get a canonical isomorphism 
\begin{equation} \label{eq:main-result-for-Lichtenbaum-complex}
\HH^{i}(X_{\zar},\tau_{\geq s}\R\pi_{\ast}(\Z/m)_{X}(n)_{\et})\cong \HH^{i}_{i-s,nr}(X_{\et},(\Z/m)_{X}(n)_{\et}).
\end{equation}  

The cohomology sheaves of the complex $(\Z/m)_{X}(n)_{\zar}$ vanish in all degrees greater than $n$. Indeed, from the Bockstein long exact sequence $$\cdots\longrightarrow\mathcal{H}^{i}(\Z(n)_{\zar})\overset{\times m}{\longrightarrow}\mathcal{H}^{i}(\Z(n)_{\zar})\longrightarrow \mathcal{H}^{i}(\Z/m(n)_{\zar})\longrightarrow\mathcal{H}^{i+1}(\Z(n)_{\zar})\overset{\times m}{\longrightarrow} \cdots,$$ we see that it is enough to show that $\mathcal{H}^{i}(\Z(n)_{\zar})=0$ for all $i>n$. The latter is implied by the Gersten conjecture \cite[Theorem 10.1]{bloch-motivic} and the fact that $\CH^{n}(\Spec k(x), 2n-i)=0$ for all $i>n$ and all points $x\in X$. Thus the canonical map $(\Z/m)_{X}(n)_{\zar}\to \R\pi_{\ast}(\Z/m)_{X}(n)_{\et}$ factors through the truncation $\tau_{\leq n }\R\pi_{\ast}\mu_{m}^{\otimes n}$. 
By the Beilinson-Lichtenbaum conjecture proven by Voevodsky \cite{Voe-milnor, Voevodsky}, we get in turn a natural quasi-isomorphism 
\begin{align} \label{eq:truncation-iso-Beilinson-Lichtenbaum-conj}
(\Z/m)_{X}(n)_{\zar}\cong \tau_{\leq n}\R\pi_{\ast}(\Z/m)_{X}(n)_{\et}.
\end{align}
Indeed, by the Chinese remainder theorem, it suffices to prove the above for $m$ a power of a prime number. In the case $m$ is coprime to the characteristic, the claim follows from \cite[Corollary 1.2]{geisser-levine} and the Bloch-Kato conjecture \cite{Voevodsky} (see also \cite[Theorem 6.6]{Voe-milnor}) whereas if $m$ is a power of the characteristic from \cite[Theorem 8.5]{geisser-levine-inventiones}.

We now consider the canonical exact triangle 
\begin{equation*}
\tau_{\leq n }\R\pi_{\ast}(\Z/m)_{X}(n)_{\et}\longrightarrow \R\pi_{\ast}(\Z/m)_{X}(n)_{\et}\longrightarrow \tau_{\geq n+1}\R\pi_{\ast}(\Z/m)_{X}(n)_{\et}\overset{+1}{\longrightarrow}
\end{equation*} 
in $D(X_{\zar},\Z)$.
The result follows then if  we apply $\R\Gamma(X_{\zar},-)$ to this triangle, consider the cohomology sequence of the resulting triangle and use \eqref{eq:main-result-for-Lichtenbaum-complex} and \eqref{eq:truncation-iso-Beilinson-Lichtenbaum-conj}.
This concludes the proof of the corollary.
\end{proof}

\subsection{Comparison to the cycle class map on higher Chow groups with integral coefficients}

\begin{proof}[Proof of Corollary \ref{cor:motivic-integral}] Let $\pi:X_{\et}\to X_{\zar}$ be the canonical map of sites. We have a natural exact triangle on $X_{\zar}$
\begin{equation}\label{eq:exact-triangle-trunc-Z}
    \tau_{\leq i+1}\R\pi_{\ast}\Z_{X}(i)_{\et}\longrightarrow \R\pi_{\ast}\Z_{X}(i)_{\et} \longrightarrow \tau_{\geq i+2}\R\pi_{\ast}\Z_{X}(i)_{\et} \stackrel{+1}\longrightarrow
\end{equation}
and we wish to show that the adjunction map $\Z_{X}(i)_{\zar}\to \R\pi_{\ast}\Z_{X}(i)_{\et}$ induces a quasi-isomorphism \begin{equation}\label{eq:quasi-iso_motivic}\Z_{X}(i)_{\zar}\cong \tau_{\leq i+1}\R\pi_{\ast}\Z_{X}(i)_{\et}.\end{equation} As in the proof of Corollary \ref{cor:motivic} the latter map factors through $\tau_{\leq i+1}\R\pi_{\ast}\Z_{X}(i)_{\et}$ simply because $\mathcal{H}^{j}(\Z_{X}(i)_{\zar})=0$ for all $j>i$. Next we consider the following commutative diagram of Zariski sheaves
\begin{equation*}
    \begin{tikzcd}
\cdots \arrow[r] & \mathcal{H}^{j-1}(\Q_{X}(i)_{\zar}) \arrow[r] \arrow[d, equal] & \mathcal{H}^{j-1}((\Q/\Z)_{X}(i)_{\zar}) \arrow[r] \arrow[d] & \mathcal{H}^{j}(\Z_{X}(i)_{\zar}) \arrow[r] \arrow[d] & \mathcal{H}^{j}(\Q_{X}(i)_{\zar}) \arrow[r] \arrow[d, equal] & \cdots \\
\cdots \arrow[r] & \R^{j-1}\pi_{\ast}\Q_{X}(i)_{\et} \arrow[r]                                  & \R^{j-1}\pi_{\ast}(\Q/\Z)_{X}(i)_{\et} \arrow[r]             & \R^{j}\pi_{\ast}\Z_{X}(i)_{\et} \arrow[r]             & \R^{j}\pi_{\ast}\Q_{X}(i)_{\et} \arrow[r]                                  & \cdots,
\end{tikzcd}
\end{equation*}
where we note that the horizontal sequences are exact. 
Since Lichtenbaum and motivic cohomology agree with rational coefficients, we find that $\mathcal{H}^{j}(\Q_{X}(i)_{\zar})\cong \R^{j}\pi_{\ast}\Q_{X}(i)_{\et}$ for all $j$. 
Moreover, the canonical maps $\mathcal{H}^{j}((\Q/\Z)_{X}(i)_{\zar})\to \R^{j}\pi_{\ast}(\Q/\Z)_{X}(i)_{\et}$ are isomorphisms as long as $i\geq j$ and injective if $j=i+1$. 
Indeed this is a formal consequence of the Beilinson-Lichtenbaum conjecture, as we will explain now for convenience of the reader.
The Chinese remainder theorem allows one to write $(\Q/\Z)_{X}(i)_{\nu}=\bigoplus_{\ell}(\Q_{\ell}/\Z_{\ell})_{X}(i)_{\nu},\ \nu\in\{\et,\zar\}$. 
Then, away from the characteristic, the claim follows from \cite[Corollary 1.2]{geisser-levine} and the Bloch-Kato conjecture proven by Voevodsky \cite{Voevodsky}.
Moreover, if $\ell=\Char(k)>0$, then the claim follows from \cite[Theorem 8.5]{geisser-levine-inventiones}. 
A simple diagram chase now gives that $\mathcal{H}^{j}(\Z_{X}(i)_{\zar})\cong \R^{j}\pi_{\ast}\Z_{X}(i)_{\et}$ for all $j\leq i+1,$ as we wanted.
\par 

Next we note that the isomorphism $\mathcal{H}^{j}(\Q_{X}(i)_{\zar})\cong \R^{j}\pi_{\ast}\Q_{X}(i)_{\et}$ implies that $\R^{j}\pi_{\ast}\Q_{X}(i)_{\et}=0$ for all $j>i$. Hence the bottom long exact sequence yields isomorphisms $\R^{j-1}\pi_{\ast}(\Q/\Z)_{X}(i)_{\et}\cong \R^{j}\pi_{\ast}\Z_{X}(i)_{\et}$ for all $j>i+1$. 
In particular, we obtain a quasi-isomorphism \begin{equation}
\label{eq:quasi-isom}\tau_{\geq i+2}\R\pi_{\ast}\Z_{X}(i)_{\et}\cong (\tau_{\geq i+1}\R\pi_{\ast}(\Q/\Z)_{X}(i)_{\et})[-1].
\end{equation}
The quasi-isomorphisms \eqref{eq:quasi-iso_motivic} and \eqref{eq:quasi-isom} imply in turn that the exact triangle \eqref{eq:exact-triangle-trunc-Z} can be rewritten as: 
$$\Z_{X}(i)_{\zar}\longrightarrow \R\pi_{\ast}\Z_{X}(i)_{\et} \longrightarrow (\tau_{\geq i+1}\R\pi_{\ast}(\Q/\Z)_{X}(i)_{\et})[-1] \stackrel{+1}\longrightarrow.$$ 
The result follows now if we apply $\R\Gamma(X_{\zar},-)$ to the latter triangle, consider the cohomology sequence of the resulting triangle and use the isomorphism $$H^{j-1}_{j-i-2}(X_{\et},(\Q/\Z)_{X}(i)_{\et} )\cong H^{j-1}(X_{\zar},\tau_{\geq i+1}\R\pi_{\ast}(\Q/\Z)_{X}(i)_{\et})$$ of Theorem \ref{thm:main-intro} for the examples \eqref{item:3:log-de-Rham-Witt} and \eqref{item:7:etale-motivic-cohomology}. \end{proof}

We note that the exact sequence of Corollary \ref{cor:motivic-integral} recovers \cite[Proposition
2.9 and Remarques 2.10 (2)]{kahn} if one sets $i=2$ and $i=3$ to the following sequence:

\begin{corollary}\label{cor:simple-version-motivic-integral} Let $X$ be a smooth and equi-dimensional algebraic scheme over a perfect field $k$. Then the sequence \eqref{eq:les-integral} gives the following exact sequence:
\begin{align*}\begin{split}
   & 0\longrightarrow H^{i+2}_{M}(X,\Z(i))\longrightarrow H^{i+2}_{L}(X,\Z(i))\longrightarrow H^{i+1}_{0,nr}(X_{\et},(\Q/\Z)_{X}(i)_{\et})\longrightarrow\cdots \\ 
   &\longrightarrow H^{2i}_{M}(X,\Z(i))\longrightarrow H^{2i}_{L}(X,\Z(i))\longrightarrow H^{2i-1}_{i-2,nr}(X_{\et},(\Q/\Z)_{X}(i)_{\et})\longrightarrow 0.
\end{split}
\end{align*}
\end{corollary}
\begin{proof} This is an immediate consequence of Corollary \ref{cor:motivic-integral} if one uses that $H^{j}_{m,nr}(X_{\et},(\Q/\Z)_{X}(i)_{\et})=0$ for $m<0$ as well as the vanishing of the motivic cohomology $H^{j}_{M}(X,\Z(i))$ for $j>2i$. The latter follows from the comparison between motivic cohomology and higher Chow groups (see \cite[Corollary 2]{voe}) where we also use that $\CH^{i}(X,m)=0$ for $m<0$.\end{proof}

\begin{corollary}\label{cor:consequence-integral-les} Let $X$ be a smooth and equi-dimensional algebraic scheme over a perfect field $k$. If either $j>2i-1$ or $j>i+\dim X-1$, then there is a natural isomorphism \begin{equation*}
 H^{j+1}_{L}(X,\Z(i))\cong H^{j}_{j-i-1,nr}(X_{\et},\Q/\Z(i)).
\end{equation*}
In particular, for $j=2i$, the above yields an isomorphism \begin{equation*}
    \Br^{i}(X)\cong H^{2i}_{i-1,nr}(X_{\et},(\Q/\Z)_{X}(i)_{\et}),
\end{equation*}
where the group $\Br^{i}(X):=H_{L}^{2i+1}(X,\Z(i))$ is the $i$-th higher Brauer group of $X$ (see e.g.\ \cite{RS-JIMJ}).
\end{corollary}
\begin{proof}[Proof of Corollary \ref{cor:consequence-integral-les}]
    By Corollary \ref{cor:motivic-integral} we find an exact sequence 
    \begin{equation*}
       \cdots\longrightarrow H^{j+1}_{M}(X,\Z(i))\longrightarrow H^{j+1}_{L}(X,\Z(i)) \longrightarrow H^{j}_{j-i-1}(X_{\et},(\Q/\Z)_{X}(i)_{\et})\to H^{j+2}_{M}(X,\Z(i)) \longrightarrow \cdots.
    \end{equation*}
   Thus to establish the desired isomorphisms, it is enough to show that $H^{j}_{M}(X,\Z(i))=0$ either if $j>2i$ or $j>i+\dim X$. Recall once again that $H^{j}_{M}(X,\Z(i))=\CH^{i}(X,2i-j)$ by \cite[Corollary 2]{voe}. Since $\CH^{i}(X,m)=0$ either if $m<0$ or $i>\dim X+m$, the vanishings in question follow. 
\end{proof}
\subsection{Refined unramified cohomology in the log de Rham Witt case}

\begin{proof}[Proof of Corollary \ref{cor:log-de-Rham-Witt}] By Theorem \ref{thm:main-intro}, we have a canonical isomorphism $$H^{i}_{j,nr}(X_{\et},(\Z/p^{r})_{X}(n))\cong H^{i}(X_{\zar},\tau_{\geq i-j}\R\pi_{\ast}(\Z/p^{r})_{X}(n)).$$ From \cite[Corollaire 1.5]{gros-suwa} we find that $\R^{q}\pi_{\ast}(\Z/p^{r})_{X}(n)=0$ for all $q\neq n, n+1$. This implies in turn that for $n\geq i-j$ the canonical map $\R\pi_{\ast}(\Z/p^{r})_{X}(n)\to\tau_{\geq i-j}\R\pi_{\ast}(\Z/p^{r})_{X}(n)$ is a quasi-isomorphism as well as for $n<i-j-1$ the truncated complex $\tau_{\geq i-j}\R\pi_{\ast}(\Z/p^{r})_{X}(n)$ is quasi-isomorphic to zero. Finally if $n=i-j-1,$ then the $(p,q)$ terms with $q\neq i-j$ of the spectral sequence $$E^{p,q}_{2}:=H^{p}(X_{\zar},\mathcal{H}^{q}(\tau_{\geq i-j}\R\pi_{\ast}(\Z/p^{r})_{X}(n)))\implies H^{p+q}(X_{\zar},\tau_{\geq i-j}\R\pi_{\ast}(\Z/p^{r})_{X}(n))$$ all vanish and thus $H^{i}_{j,nr}(X,(\Z/p^{r})_{X}(n))\cong H^{j}(X_{\zar}, \R^{i-j}\pi_{\ast}(\Z/p^{r})_{X}(n))$ as we want. This finishes the proof.\end{proof}

\begin{remark} The Leray spectral sequence associated to $\pi: X_{\et}\to X_{\zar}$, 
$$
E^{p,q}_{2}:=H^{p}(X_{\zar},\R^{q}\pi_{\ast}(\Z/p^{r})_{X}(n))\implies H^{p+q}(X,(\Z/p^{r})_{X}(n)),
$$ 
together with the fact that $\R^{q}\pi_{\ast}(\Z/p^{r})_{X}(n)=0$ for $q\neq n,n+1$ (see \cite[Corollaire 1.5]{gros-suwa}) yield an exact sequence \begin{align}\label{eq:seq-spectral}\begin{split}
 & 0\longrightarrow H^{1}(X_{\zar},\R^{n}\pi_{\ast}(\Z/p^{r})_{X}(n))\longrightarrow H^{n+1}(X,(\Z/p^{r})_{X}(n))\\ 
 & \longrightarrow H^{0}(X_{\zar}, \R^{n+1}\pi_{\ast}(\Z/p^{r})_{X}(n))\longrightarrow \cdots \longrightarrow H^{n}(X_{\zar}, \R^{n}\pi_{\ast}(\Z/p^{r})_{X}(n))\\
 & \longrightarrow H^{2n}(X,(\Z/p^{r})_{X}(n)) \longrightarrow H^{n-1}(X_{\zar},\R^{n+1}\pi_{\ast}(\Z/p^{r})_{X}(n))\longrightarrow 0.\end{split} \end{align} Since by \eqref{eq:truncation-iso-Beilinson-Lichtenbaum-conj} we also have $H^{i}_{M}(X,\Z/p^{r}(n))\cong H^{i-n}(X_{\zar}, \R^{n}\pi_{\ast}(\Z/p^{r})_{X}(n)),$ we find from Corollary \ref{cor:log-de-Rham-Witt} that the sequence \eqref{eq:seq-spectral} identifies with the one of Corollary \ref{cor:motivic}. 
\end{remark}

\subsection{Pullbacks for refined unramified cohomology}

\begin{proof}[Proof of Corollary \ref{cor:functoriality}]
Let $f:X\to Y$ be a morphism between equi-dimensional smooth algebraic $k$-schemes and let $k$, $\nu$ and $K_Y^\bullet\in D(Y_\nu,\Z)$ be as in one of the examples \eqref{item:1:etale-coho}--\eqref{item:7:etale-motivic-cohomology} above.  
Let $K_X^\bullet\in D(X_\nu,\Z)$ be the analogous complex on $X_\nu$.
We aim to construct functorial pullback maps 
$$
f^\ast:H^i_{j,nr}(Y_\nu ,K_Y^\bullet)\longrightarrow H^i_{j,nr}(X_\nu , K_X^\bullet).
$$
To this end, let us first construct a natural map $f^{\ast}:H^{i}_{j,nr}(Y_{\nu},K^{\bullet}_{Y})\to H^{i}_{j,nr}(X_{\nu},f^{\ast}K^{\bullet}_{Y})$. By Theorem \ref{thm:main-intro}, this is equivalent to establishing a natural map \begin{equation}\label{eq:f^{*}} f^{\ast}:H^{i}(Y_{\zar},\tau_{\geq i-j}\R\pi_{Y\ast}K^{\bullet}_{Y})\longrightarrow H^{i}(X_{\zar},\tau_{\geq i-j}\R\pi_{X\ast}f^{\ast}K^{\bullet}_{Y}),\end{equation} where $\pi_{Y}:Y_{\nu}\to Y_{\zar}$ (resp. $\pi_{X}:X_{\nu}\to X_{\zar}$) denotes the natural morphism of sites. The unit $\alpha_{f}:\id\to\R f_{\ast}f^{\ast}$ of the adjunction $(f^{\ast},f_{\ast})$ gives a map \begin{equation*}\alpha_{f}:\tau_{\geq i-j}\R\pi_{Y\ast}K^{\bullet}_{Y}\longrightarrow\R f_{\ast}f^{\ast}\tau_{\geq i-j}\R\pi_{Y\ast}K^{\bullet}_{Y}.\end{equation*} Since $f^{\ast}:Shv(Y_{\zar},\Z)\to Shv(X_{\zar},\Z)$ is an exact functor, it induces a functor on the derived level that we denote by the same symbol $f^{\ast}$ and which commutes with truncation. In particular, the target of $\alpha_{f}$ is $\R f_{\ast}\tau_{\geq i-j}f^{\ast}\R\pi_{Y\ast}K^{\bullet}_{Y}$. We exhibit a natural map $$\beta: f^{\ast}\R\pi_{Y\ast}K^{\bullet}_{Y}\longrightarrow \R\pi_{X\ast}f^{\ast}K^{\bullet}_{Y}.$$ By adjunction the latter is equivalent to constructing a map $$\beta^{\#}: \R\pi_{Y\ast}K^{\bullet}_{Y}\longrightarrow \R f_{\ast}\R\pi_{X\ast}f^{\ast}K^{\bullet}_{Y}.$$ Note that $\R f_{\ast}\R\pi_{X\ast}=\R (f\circ \pi_{X})_{\ast}=\R (\pi_{Y}\circ f)_{\ast}=\R\pi_{Y\ast}\R f_{\ast}$ and hence the target of $\beta^{\#}$ is really $\R\pi_{Y\ast} \R f_{\ast}f^{\ast}K^{\bullet}_{Y}$. To obtain the map $\beta^{\#}$, we simply apply now the derived functor $\R\pi_{Y\ast}$ to the natural map $K^{\bullet}_{Y}\to \R f_{\ast}f^{\ast}K^{\bullet}_{Y}$ obtained by evaluating the unit of the adjunction $(f^{\ast},f_{\ast})$ at $K^{\bullet}_{Y}$. 

\par 
    Consider the composite $$\tau_{\geq i-j}\R\pi_{Y\ast}K^{\bullet}_{Y}\overset{\alpha_{f}}{\longrightarrow}\R f_{\ast}\tau_{\geq i-j}f^{\ast}\R\pi_{Y\ast}K^{\bullet}_{Y}\overset{\R f_{\ast}\tau_{\geq i-j}\beta}\longrightarrow \R f_{\ast}\tau_{\geq i-j} \R\pi_{X\ast}f^{\ast}K^{\bullet}_{Y}.$$ We then obtain the map \eqref{eq:f^{*}} as desired by applying $\R^i\Gamma(Y_{\zar},-)$ to the above composition.
    
\par 
 Next, we need to show that there exists a natural map $f^{\ast}K^{\bullet}_{Y}\to K^{\bullet}_{X}$ for each one of the examples \eqref{item:1:etale-coho}--\eqref{item:7:etale-motivic-cohomology}. The examples \eqref{item:1:etale-coho} and \eqref{item:5:proetal} are of the same nature i.e. the complexes in question are given by $K^{\bullet}_{X}:=p_{X}^{\ast}K^{\bullet}$ and $K^{\bullet}_{Y}:=p_{Y}^{\ast}K^{\bullet},$ where $p_{X}: X\to\Spec k$ (resp. $p_{Y}: Y\to\Spec k$) is the structure morphism of $X$ (resp. $Y$) and where $K^{\bullet}\in D(\Spec(k)_{\nu})$. Since $f^{\ast}p_{Y}^{\ast}\cong p_{X}^{\ast}$, the claim follows. 
 For \eqref{item:2:Betti-coho}, note that any complex of constant sheaves of abelian groups in $D(X_{\an})$ is a pullback of a complex from $D(\Spec(\C)_{\an})=D(\text{Ab})$. 
 Thus the same reasoning as before implies the result. For \eqref{item:4:algebraic-de-Rham}, see \cite[\href{https://stacks.math.columbia.edu/tag/0FKL}{Tag 0FKL}]{stacks-project} and for \eqref{item:3:log-de-Rham-Witt} and \eqref{item:7:de-Rham-Witt}, see \cite[pp. 10, (1.2.2)]{gros} and \cite[pp. 548, (1.12.3)]{illusie}, respectively.\par

It remains to treat the case of items \eqref{item:6:motivic-cohomology} and \eqref{item:7:etale-motivic-cohomology}. 
As in the proof of Theorem \ref{thm:main-intro}, recall that by \cite[Proposition 12.1]{F-S}, there is a canonical quasi-isomorphism of complexes of Zariski sheaves on the small site of \'etale $Y$-schemes: $$\Z^{SF}(n)|_{Y_{\et}}\otimes A \cong A_{Y}(n),$$ where $\Z^{SF}(n)=\underline{C}_{\ast}(z_{equi}(\mathbb{A}^{n},0))[-2n]$ is the motivic complex of Friedlander and Suslin given in \cite[Section 8]{F-S}.  
As before, the presheaf $z_{equi}(X,0)$ on $Sm/k$ is the \'etale sheaf of equi-dimensional cycles on $X$ of relative dimension $0$ as defined in \cite[Chapter 4, $\S$2, pp. 141]{v-s-f}, whereas the complex $\underline{C}_{\ast}(F)$ associated to a presheaf $F$ of abelian groups on $Sm/k$ is the singular simplicial complex of $F$, see \cite[Chapter 4, $\S 4$, pp. 150]{v-s-f}. In particular, we can work with the motivic complex $\Z^{SF}(n)$ instead of Bloch's cycle complex. It is readily seen that by contravariant functoriality of $z_{equi}(X,0),$ we obtain maps $$\underline{C}_{\ast}(z_{equi}(\mathbb{A}^{n},0))|_{Y_{\nu}}\longrightarrow f_{\ast}\underline{C}_{\ast}(z_{equi}(\mathbb{A}^{n},0))|_{X_{\nu}}$$ and especially a natural map $\Z^{SF}(n)|_{Y_{\nu}}\to f_{\ast}\Z^{SF}(n)|_{X_{\nu}}$. The adjoint of the latter gives naturally what we want.

\par Lastly, note that we have a commutative diagram 
\begin{equation*}
\begin{tikzcd}
\R\pi_{Y\ast}K^{\bullet}_{Y} \arrow[r, "\alpha_{f}"] \arrow[d]      & \R f_{\ast}f^{\ast}\R\pi_{Y\ast}K^{\bullet}_{Y} \arrow[r, "\R f_{\ast}\beta"] \arrow[d]                     &  \R f_{\ast}\R\pi_{X\ast}f^{\ast}K^{\bullet}_{Y} \arrow[d] \arrow[r]      &  \R f_{\ast}\R\pi_{X\ast}K^{\bullet}_{X} \arrow[d]      \\
\tau_{\geq i-j}\R\pi_{Y\ast}K^{\bullet}_{Y} \arrow[r, "\alpha_{f}"] & \R f_{\ast}\tau_{\geq i-j}f^{\ast}\R\pi_{Y\ast}K^{\bullet}_{Y} \arrow[r, "\R f_{\ast}\tau_{\geq i-j}\beta"] & \R f_{\ast}\tau_{\geq i-j} \R\pi_{X\ast}f^{\ast}K^{\bullet}_{Y} \arrow[r] &  \R f_{\ast}\tau_{\geq i-j}\R\pi_{X\ast}K^{\bullet}_{X},
\end{tikzcd}
\end{equation*}
where the vertical maps are the canonical ones. Applying $\R\Gamma(Y_{\zar},-)$ to the above diagram and taking cohomology, we obtain a commutative diagram: \begin{equation*}
    \begin{tikzcd}
{H^{i}(Y_\nu ,K^{\bullet}_{Y})} \arrow[r, "f^{\ast}"] \arrow[d]                   & {H^{i}(X_\nu ,K^{\bullet}_{X})} \arrow[d]                   \\
{H^{i}(Y_\zar,\tau_{\geq i-j}\R\pi_{Y\ast}K^{\bullet}_{Y})} \arrow[r, "f^{\ast}"] & {H^{i}(X_\zar,\tau_{\geq i-j}\R\pi_{X\ast}K^{\bullet}_{X})}.
\end{tikzcd}
\end{equation*}
The vertical maps identify by Theorem \ref{thm:main-intro} with the natural maps $H^{i}(Y_\nu ,K^{\bullet}_{Y})\to H^{i}_{j,nr}(Y_\nu,K^{\bullet}_{Y})$ and $H^{i}(X_\nu,K^{\bullet}_{X})\to H^{i}_{j,nr}(X_\nu,K^{\bullet}_{X}),$ respectively.
This concludes the proof of Corollary \ref{cor:functoriality}, as we want.
\end{proof}

\subsection{Action of correspondences}
In this section we prove Corollary \ref{cor:action-correspondences}. Let $X$ be a smooth and equi-dimensional algebraic scheme over a perfect field $k$ of characteristic $p>0$. Recall first that by works of Berthelot (see \cite[Chapter VI, \S 3]{berthelot}) and  Gros (see \cite[D\'efinition 4.1.7]{gros}), there is a well-defined cycle class map 
\begin{equation}\label{eq:cycle-class-map-crystaline}
    \cl^{c}_{X}:\CH^{c}(X)\longrightarrow H^{2c}(X,W\Omega^{\bullet}_{X})
\end{equation}
that satisfies various natural compatibility properties. 

We start with some preliminary lemmas. 

\begin{lemma}\label{lem:action-cycles} Let $X$ be a smooth equi-dimensional algebraic scheme over a perfect field $k$ of characteristic $p>0$. 
Then for any pair of smooth and equi-dimensional algebraic $k$-schemes $X$ and $Y$, there is a natural bi-additive pairing 
\begin{equation}\label{eq:action-cycles}
    \times :\CH^{c}(X)\times H^{i}_{j,nr}(Y,W\Omega^{\bullet}_{Y})\longrightarrow H^{i+2c}_{j+c,nr}(X\times Y,W\Omega^{\bullet}_{X\times Y}),
\end{equation}
which is compatible with the pairing 
\begin{equation*}
    \CH^{c}(X)\times H^{i}(Y,W\Omega^{\bullet}_{Y})\longrightarrow H^{i+2c}(X\times Y,W\Omega^{\bullet}_{X\times Y}),\ [\Gamma]\otimes \alpha \mapsto p^{\ast}\cl^{c}_{X}([\Gamma])\cup q^{\ast} \alpha,
\end{equation*}
where $p:X\times Y\to X$ and $q: X\times Y\to Y$ are the canonical projections i.e. the following diagram commutes:\begin{equation*}
    \begin{tikzcd}
{\CH^{c}(X)\times H^{i}(Y,W\Omega^{\bullet}_{Y})} \arrow[r] \arrow[d] & {H^{i+2c}(X\times Y,W\Omega^{\bullet}_{X\times Y})} \arrow[d] \\
{\CH^{c}(X)\times H^{i}_{j,nr}(Y,W\Omega^{\bullet}_{Y})} \arrow[r]    & {H^{i+2c}_{j+c,nr}(X\times Y,W\Omega^{\bullet}_{X\times Y}).} 
\end{tikzcd}
\end{equation*}
\end{lemma}
\begin{proof} Let $\mathcal{Z}^{c}(X)$ denote the free abelian group generated by integral closed subschemes of $X$ of codimension $c$. As a first step, we construct a bi-additive pairing 
\begin{equation}
    \mathcal{Z}^{c}(X)\times  H^{i}(F_{j}Y,W\Omega^{\bullet}_{Y})\longrightarrow H^{i+2c}(F_{j+c}(X\times Y),W\Omega^{\bullet}_{X\times Y}).
\end{equation}
Pick a cycle $\Gamma\in\mathcal{Z}^{c}(X)$ and denote its support by $|\Gamma|:=\text{Supp}(\Gamma)\subset X$. Let $\alpha\in H^{i}(F_{j}Y,W\Omega^{\bullet}_{Y})$ be any class and choose a lift $\tilde{\alpha}\in H^{i}(U,W\Omega^{\bullet}_{Y})$, where $U\subset Y$ is open with complement $R:=Y\setminus U$ of codimension $\geq j+1$. Consider the cycle class $\cl_{|\Gamma|}(\Gamma)\in H^{2c}_{|\Gamma|}(X,W\Omega^{\bullet}_{X})$ and note that it behaves well with respect to pull-backs. Indeed, by \cite[D\'efinition 4.1.7]{gros}, it suffices to prove the compatibility for the cohomology class $\cl_{|\Gamma|}(\Gamma)\in H^{c}_{|\Gamma|}(X,W\Omega^{c}_{X,\text{log}})$. The natural isomorphism \cite[(4.1.6)]{gros}, allows us to remove the singular locus of $|\Gamma|$. Thus the compatibility property in question follows from the commutative diagram \cite[(3.5.20)]{gros}, once we use the identification \cite[(3.5.19)]{gros}. In particular, if $p:X\times Y\to X$ is the natural projection, then $p^{\ast}\cl_{|\Gamma|}(\Gamma)=\cl_{|\Gamma|\times Y}(p^{\ast}\Gamma)\in H^{2c}_{|\Gamma|\times Y}(X\times Y,W\Omega^{\bullet}_{X\times Y})$.\par

We consider the following composite
\begin{equation}\label{eq:composite-1}
    H^{i}(U,W\Omega^{\bullet}_{Y})\overset{\cl_{|\Gamma|\times U}(p^{\ast}\Gamma)\cup q^{\ast}}{\longrightarrow} H^{i+2c}_{|\Gamma|\times U} ((X\times Y)\setminus (|\Gamma|\times R),W\Omega^{\bullet}_{X\times Y})\overset{\iota_{\ast}}{\longrightarrow} H^{i+2c}((X\times Y)\setminus (|\Gamma|\times R),W\Omega^{\bullet}_{X\times Y}),
\end{equation} 
where for the middle term we implicitly use the isomorphism $$H^{i+2c}_{|\Gamma|\times U}(X\times U,W\Omega^{\bullet}_{X\times Y})\cong H^{i+2c}_{|\Gamma|\times U} ((X\times Y)\setminus (|\Gamma|\times R),W\Omega^{\bullet}_{X\times Y})$$ that follows from excision. The closed subset $|\Gamma|\times R\subset X\times Y$ has codimension $\geq j+c+1$ in $X\times Y$ and thus we can define $\Gamma \times \alpha\in H^{i+2c}(F_{j+c}(X\times Y),W\Omega^{\bullet}_{X\times Y})$ as the class represented by the image of $\tilde{\alpha}$ via \eqref{eq:composite-1}. Note that the class $\Gamma \times \alpha$ is clearly independent of our choice of a lift $\tilde{\alpha}$, as we can always shrink $U\subset X$ to an open subset $F_{j}Y\subset U'\subset U$ over which a given lift of $\alpha$ agrees with $\tilde{\alpha}|_{U'}$.\par 
By construction, the diagram \begin{equation*} \begin{tikzcd}
\mathcal{Z}^{c}(X)\times H^{i}(F_{j+1}Y,W\Omega^{\bullet}_{Y}) \arrow[d] \arrow[r] & H^{i}(F_{j+c+1}(X\times Y),W\Omega^{\bullet}_{X\times Y}) \arrow[d] \\
\mathcal{Z}^{c}(X)\times H^{i}(F_{j}Y,W\Omega^{\bullet}_{Y})               \arrow[r] & H^{i}(F_{j+c}(X\times Y),W\Omega^{\bullet}_{X\times Y})            
\end{tikzcd}
\end{equation*}
commutes, thus inducing a bi-additive pairing $$ \times :\mathcal{Z}^{c}(X)\times H^{i}_{j,nr}(Y,W\Omega^{\bullet}_{Y})\longrightarrow H^{i+2c}_{j+c,nr}(X\times Y,W\Omega^{\bullet}_{X\times Y}).$$\par

To get \eqref{eq:action-cycles}, it remains to show that the class $\Gamma\times\alpha$ vanishes if $\Gamma$ is rationally equivalent to zero. Indeed, in this case there exists an equi-dimensional closed subset $Z\subset X$ of codimension $c-1$, such that $|\Gamma|\subset Z$ and $\cl_{Z}(\Gamma)=0\in H^{2c}_{Z}(X,W\Omega^{\bullet}_{X})$. Let $\tilde{\alpha}\in H^{i}(V,W\Omega^{\bullet}_{Y})$ be a lift of $\alpha\in H^{i}_{j,nr}(Y,W\Omega^{\bullet}_{Y})$ over an open subset $V\subset Y$ whose complement $D:=Y\setminus V$ has codimension $\geq j+2$ and consider the composite 
\begin{equation}\label{eq:composite-2}
    H^{i}(V,W\Omega^{\bullet}_{Y})\overset{\cl_{Z\times V}(p^{\ast}\Gamma)\cup q^{\ast}}{\longrightarrow} H^{i+2c}_{Z\times V} ((X\times Y)\setminus (Z\times D),W\Omega^{\bullet}_{X\times Y})\overset{\iota_{\ast}}{\longrightarrow} H^{i+2c}((X\times Y)\setminus (Z\times D),W\Omega^{\bullet}_{X\times Y}),
\end{equation}
where as before we make use of the excision isomorphism $$H^{i+2c}_{Z\times V} ((X\times Y)\setminus (Z\times D),W\Omega^{\bullet}_{X\times Y})\cong H^{i+2c}_{Z\times V}(X\times V,W\Omega^{\bullet}_{X\times Y})$$ for the middle term.\par
Note that the closed subset $Z\times D\subset X\times Y$ has codimension $\geq j+c+1$ and it is readily seen that the image of $\tilde{\alpha}\in H^{i}(V,W\Omega^{\bullet}_{Y})$ via \eqref{eq:composite-2} gives a representative for the class $\Gamma\times \alpha$. Since the second map in \eqref{eq:composite-2} is zero ($\cl_{Z}(\Gamma)=0$), we find that $\Gamma\times \alpha=0$, as claimed. 
\end{proof}

\begin{lemma}\label{lem:pushforward} Let $k$ be a perfect field of characteristic $p>0$. Let $f:X\to Y$ be a proper morphism between smooth and equi-dimensional algebraic schemes over $k$ and set $r:=\dim Y-\dim X$. Then for any integers $i,j\geq 0,$ there is a natural pushforward map $f_{\ast}:H^{i}_{j,nr}(X,W\Omega^{\bullet}_{X})\to H^{i+2r}_{j+r,nr}(Y,W\Omega^{\bullet}_{Y})$ that is compatible with $f_{\ast}:H^{i}(X,W\Omega^{\bullet}_{X})\to H^{i+2r}(Y,W\Omega^{\bullet}_{Y})$ i.e. the following diagram commutes:
\begin{equation*}
\begin{tikzcd}
{H^{i}(X,W\Omega^{\bullet}_{X})} \arrow[r, "f_{\ast}"] \arrow[d] & {H^{i+2r}(Y,W\Omega^{\bullet}_{Y})} \arrow[d] \\
{H^{i}_{j,nr}(X,W\Omega^{\bullet}_{X})} \arrow[r, "f_{\ast}"]    & { H^{i+2r}_{j+r,nr}(Y,W\Omega^{\bullet}_{Y}).}
\end{tikzcd}
\end{equation*}
\end{lemma}
\begin{proof} It is enough to construct for $j\geq0$ natural maps $f_{\ast}:H^{i}(F_{j}X,W\Omega^{\bullet}_{X})\to H^{i+2r}(F_{j+r}Y,W\Omega^{\bullet}_{Y})$ that fit into a commutative diagram 
\begin{equation}\label{eq:compatibility}
\begin{tikzcd}
H^{i}(F_{j+1}X,W\Omega^{\bullet}_{X}) \arrow[r, "f_{\ast}"] \arrow[d] & H^{i+2r}(F_{j+1+r}Y,W\Omega^{\bullet}_{Y}) \arrow[d] \\
H^{i}(F_{j}X,W\Omega^{\bullet}_{X}) \arrow[r, "f_{\ast}"]             & H^{i+2r}(F_{j+r}Y,W\Omega^{\bullet}_{Y}). \end{tikzcd}
\end{equation}

By \cite[Definition 1.2.1]{gros}, there is a natural morphism 
\begin{equation}\label{eq:derived-pushforward} f_{\ast}:\R f_{\ast}W\Omega^{\bullet}_{X}\longrightarrow W\Omega^{\bullet}_{Y}(r)[r]\end{equation} 
in $D^{+}(X_{\zar},R)$, i.e.\ in the derived category of Zariski sheaves of graded modules over the Raynaud ring $R$, see \cite[(2.1)]{illusie-book}. Here $W\Omega^{\bullet}_{Y}(r)$ is the graded $R$-module deduced by the usual shift of degrees i.e. $(W\Omega^{\bullet}_{Y}(r))^{a}=W\Omega^{a+r}_{Y}$ and the change of the differential $d$ to $(-1)^{r}d$. There is a derived functor $\R\Gamma:D^{+}(X_{\zar},R)\to D^{+}(R)$ for the global section functor $\Gamma$. As we view complexes in $D^{+}(R)$ (resp. $D^{+}(X_{\zar},R)$) as double complexes whose rows correspond to graded $R$-modules, we also have a functor $\underline{s}:D^{+}(R)\to D^{+}(W)$ (resp. $\underline{s}:D^{+}(X_{\zar},R)\to D^{+}(X_{\zar},W)$) that maps a double complex to its total complex, see \cite[(2.1)]{illusie-book}.\par 

Applying $\underline{s}\circ\R\Gamma$ to \eqref{eq:derived-pushforward} gives $$\underline{s}\R\Gamma(X,W\Omega^{\bullet}_{X})\longrightarrow \underline{s}\R\Gamma(Y,W\Omega^{\bullet}_{Y})[2r],$$ where we implicitly utilise the isomorphism of total complexes $\underline{s}\R\Gamma(Y,W\Omega^{\bullet}_{Y})[2r]\cong \underline{s}\R\Gamma(Y,W\Omega^{\bullet}_{Y}(r)[r])$ defined degree-wise on the summands by $$(-1)^{pr}\id:\R\Gamma(Y,W\Omega^{p}_{Y})^{q}\longrightarrow \R\Gamma(Y,W\Omega^{p}_{Y})^{q}.$$ Thus taking cohomology in turn yields a natural map $f_{\ast}:H^{i}(X,W\Omega^{\bullet}_{X})\to H^{i+2r}(Y,W\Omega^{\bullet}_{Y})$ for all $i$.\par

Next pick a closed subset $Z\subset X$ with $\dim X-\dim Z\geq j+1$ and set $U:=X\setminus Z$. 
Then we have $\dim Y-\dim f(Z)\geq r+j+1$ and since the map $f:X\setminus f^{-1}(f(Z))\to X\setminus f(Z)$ is proper, the discussion above yields a homomorphism $$H^{i}(U,W\Omega^{\bullet}_{X})\longrightarrow H^{i}(X\setminus f^{-1}(f(Z)),W\Omega^{\bullet}_{X})\overset{f_{\ast}}\longrightarrow H^{i+2r}(Y\setminus f(Z),W\Omega^{\bullet}_{Y})\longrightarrow H^{i+2r}(F_{j+r}Y,W\Omega^{\bullet}_{Y}).$$ By taking the limit now over all opens $F_{j}X\subset U\subset X$, we deduce the desired map $$f_{\ast}:H^{i}(F_{j}X,W\Omega^{\bullet}_{X})\longrightarrow H^{i+2r}(F_{j+r}Y,W\Omega^{\bullet}_{Y}).$$ The compatibility property \eqref{eq:compatibility} for $f_{\ast}$ is clear by construction. This finishes the proof.\end{proof}

By combining Lemma \ref{lem:action-cycles} and Corollary \ref{cor:functoriality}, we are able to prove Corollary \ref{cor:action-correspondences} stated in the introduction.

\begin{proof}[Proof of Corollary \ref{cor:action-correspondences}]
 Let $k$ be a perfect field of characteristic $p>0$ and let $X$ and $Y$ be smooth, proper and equi-dimensional algebraic schemes over $k$.  
 We aim to construct a bi-additive pairing 
\begin{equation}\label{eq:action-correspondences-body}
    \CH^{c}(X\times Y)\times H^{i}_{j,nr}(X,W\Omega^{\bullet}_{X})\longrightarrow H^{i+2c-2d_{X}}_{j+c-d_{X},nr}(Y,W\Omega^{\bullet}_{Y}), \ \ \  ([\Gamma],\alpha)\mapsto [\Gamma]_{\ast}(\alpha),
\end{equation} 
where $d_{X}=\dim X$.
We define \eqref{eq:action-correspondences-body} as the composite 
\begin{equation*}
    \CH^{c}(X\times Y)\times H^{i}_{j,nr}(X,W\Omega^{\bullet}_{X})\overset{\Delta^{\ast}\circ\eqref{eq:action-cycles}\circ (\id \times p^{\ast})}{\longrightarrow} H^{i+2c}_{j+c,nr}(X\times Y,W\Omega^{\bullet}_{X\times Y})\overset{q_{\ast}}{\longrightarrow}H^{i+2c-2d_{X}}_{j+c-d_{X},nr}(Y,W\Omega^{\bullet}_{Y}),
\end{equation*}
where $p:X\times Y\to X, q:X\times Y\to Y$ are the natural projections and the morphism $\Delta: X\times Y\to (X\times Y)^{2}$ denotes the diagonal. 
Note that in the above composition the pullback maps $p^{\ast}$ and $\Delta^{\ast}$ are well-defined by Corollary \ref{cor:functoriality} and that $q_{\ast}$ is the pushforward map from Lemma \ref{lem:pushforward}.
Finally an adaptation of the argument in \cite[Corollary 6.8 (3)]{Sch-moving} implies that the construction \eqref{eq:action-correspondences} is compatible with respect to the composition of correspondences.
\end{proof}

 \section*{Acknowledgements} We thank Lin Zhou for useful comments and H\'el\`ene Esnault for a conversation. 
 We are grateful to the referee for his or her comments and to Jean-Louis Colliot-Th\'el\`ene for asking a question which led us to Corollary \ref{cor:motivic-integral}.
 This project has received funding from the European Research Council (ERC) under the European Union's Horizon 2020 research and innovation programme under grant agreement No 948066 (ERC-StG RationAlgic).



\begin{thebibliography}{HKLR} 


\bibitem[Stacks24]{stacks-project}
The Stacks project authors, {\em  The Stacks project},
 \url{https://stacks.math.columbia.edu}, 2024.


\bibitem[Ale23]{Ale-griffiths} Th.\ Alexandrou, {\em Torsion in Griffiths groups}, Preprint 2023, arXiv:2303.04083.

\bibitem[Ale24]{Ale-zero-cycles}
Th.\ Alexandrou, {\em Two Cycle Class Maps on Torsion Cycles},  arXiv:2401.11014, to appear in IMRN.

\bibitem[AS23]{alexandrou-schreieder} Th.\ Alexandrou and S.\ Schreieder, {\em On Bloch's map for torsion cycles over non-closed fields}, Forum of Mathematics, Sigma (2023), Vol. 11:e53 1--21.


\bibitem[BS15]{BS} B.\ Bhatt and P.\ Scholze, {\em The pro-\'etale topology of schemes},  Ast\'erisque \textbf{369} (2015),  99--201.

\bibitem[BO74]{BO}
S.\ Bloch and A.\ Ogus, {\em Gersten's conjecture and the homology of schemes}, Ann.\ Sci.\ \'Ec.\ Norm.\ Sup\'er., \textbf{7} (1974), 181--201.

\bibitem[Blo86]{bloch-motivic}
S.\ Bloch, {\em Algebraic cycles and higher K-theory},
Adv. in Math.,\ \textbf{61} (1986), 267--304. 
 
\bibitem[CD16]{cd}
D.-C.\ Cisinski and F.\ D\'eglise, {\em \'Etale motives}, Compositio Mathematica \textbf{152} (2016), 556--666. 
 
\bibitem[CT95]{CT}
J.-L.\ Colliot-Th\'el\`ene, {\em Birational invariants, purity and the Gersten conjecture}, K-theory and algebraic geometry: connections with quadratic forms and division algebras (Santa Barbara, CA, 1992), 1--64, Proc.\ Sympos.\ Pure Math.\, \textbf{58}, AMA, Providence,RI, 1995.

\bibitem[CTO89]{CTO}
J.-L.\ Colliot-Th\'el\`ene and M.\ Ojanguren, {\em Vari\'et\'es unirationnelles non rationnelles : au-del\`a de l'exemple d'Artin et Mumford}, Invent.\ Math.\ \textbf{97} (1989), 141--158.

\bibitem[CTHK97]{CTHK}
J.-L.\ Colliot-Th\'el\`ene, R.\ Hoobler, and B.\ Kahn, {\em  The Bloch-Ogus-Gabber theorem}, Algebraic K-theory (Toronto, ON, 1996), 31--94,  Fields Inst.\  Commun., \textbf{16}, Amer. \ Math.\  Soc., Providence, RI, 1997.
 
  
\bibitem[CTV12]{CTV}
J.-L.\ Colliot-Th\'el\`ene and C.\ Voisin, {\em Cohomologie non ramifi\'ee et conjecture de Hodge enti\`ere}, Duke Math.\ J.\ \textbf{161} (2012), 735--801.

\bibitem[CTK13]{CT-kahn}
J.-L.\ Colliot-Th\'el\`ene and B.\ Kahn {\em Cycles de codimension 2 et H3 non ramifi\'e pour les vari\'et\'es sur les corps finis}, Journal of K-Theory, \textbf{11} (2013), 1--53.

 
 
\bibitem[FS02]{F-S}
E.\ M.\ Friedlander and A.\ Suslin, {\em The spectral sequence relating algebraic K-theory to motivic cohomology}, Ann.\ Scient.\ Éc.\ Norm.\ Sup.\ \textbf{35} (2002), 773--875.

\bibitem[GL00]{geisser-levine-inventiones}
T.\  Geisser and  M.\ Levine,  {\em The K-theory of fields in characteristic p}, Invent.\  Math.\ \textbf{139} (2000), 459--493.

\bibitem[GL01]{geisser-levine}
T.\ Geisser and M.\ Levine, {\em The Bloch-Kato conjecture and a theorem of Suslin–Voevodsky}, J.\ Reine Angew.\ Math.\ \textbf{530} (2001), 55--103.

 \bibitem[G57]{grothendieck}
 A.\ Grothendieck, {\em Sur quelques points d'alg\'ebre homologique}, Tohoku Math.\ J.\ \textbf{9} (1957), 119--221.
 
 \bibitem[G85]{gros}
M.\ Gros, {\em Classes de Chern et de cycles en cohomologie de Hodge-Witt logarithmique}, M\'emoires de la SMF, $2^e$ s\'erie, tome \textbf{21} (1985).

\bibitem[GS88]{gros-suwa}
M.\ Gros and S.\ Suwa, {\em La conjecture de Gersten pour les faisceaux de Hodge-Witt logarithmiques}, Duke Math.\ J.\ \textbf{57} (1988), 615--628

\bibitem[Be74]{berthelot}
P.\ Berthelot, {\em Cohomologie cristalline des sch\'emas de caract\'eristique p > 0}, Lecture notes in Math., vol \textbf{407} (1974), Berlin-Heidelberg-New York: Springer-Verlag.

\bibitem[Il79]{illusie}
L.\ Illusie, {\em Complexe de de Rham-Witt et cohomologie cristalline}, Ann.\ Sci.\ l'\'ENS \textbf{12} (1979), 501--661.

\bibitem[Il83]{illusie-book}
L.\ Illusie, {\em Finiteness, duality, and Künneth theorems in the cohomology of the De Rham Witt complex}, In: Raynaud, M., Shioda, T. (eds) Algebraic Geometry. Lecture Notes in Mathematics, vol \textbf{1016} (1983). Springer, Berlin, Heidelberg.


\bibitem[Kah96]{kahn96}
B.\ Kahn, {\em Applications of weight two motivic cohomology}, Documenta Mathematica \textbf{1} (1996), 395–-416.

\bibitem[Kah12]{kahn}
B.\ Kahn, {\em Classes de cycles motiviques \'etales}, Algebra and Number Theory \textbf{6} (2012), 1369--1407.


\bibitem[Kok23]{kok} K.\ Kok, {\em On the failure of the integral Hodge/Tate conjecture for products with projective hypersurfaces}, Preprint 2023, arXiv:2305.08961.

\bibitem[KZ23]{kok-zhou}
K.\ Kok and L.\ Zhou, {\em Higher Chow groups with finite coefficients and refined unramified cohomology}, Preprint 2023, to appear in Advances in Mathematics.

\bibitem[KZ24]{kok-zhou2}
K.\ Kok and L.\ Zhou, {\em On the functoriuality of refined unramified cohomology}, Preprint 2024, arXiv:2403.10198. 






\bibitem[Ma17]{Ma}
S.\ Ma, {\em Torsion 1-cycles and the coniveau spectral sequence}, Documenta Math.\ \textbf{22} (2017), 1501--1517.

 

 
\bibitem[Qui73]{quillen}
D.\ Quillen,  {\em Higher algebraic K-theory, I},  Lecture Notes in Mathematics \textbf{341}, Springer, Berlin, 1973, 77--139.
 


\bibitem[RS16]{RS-JIMJ}
A.\ Rosenschon and V.\ Srinivas, {\em \'Etale motivic cohomology and algebraic cycles}, J.\ Inst.\ Math.\ Jussieu\ \textbf{15} (2016), 511--537.




\bibitem[Sch23]{Sch-refined}
S.\ Schreieder, {\em Refined unramified cohomology of schemes}, Compositio Mathematica, \textbf{159} (2023), 1466--1530.

\bibitem[Sch24]{Sch-griffiths}
S.\ Schreieder, {\em Infinite torsion in Griffiths groups},  
J.\ Eur.\ Math.\ Soc.\ (2024), pp. 1--31, DOI 10.4171/JEMS/1419.
 

\bibitem[Sch22]{Sch-moving}
S.\ Schreieder, {\em A moving lemma for cohomology with support},  Preprint 2022,  
to appear in EPIGA (special volume in honour of Claire Voisin).


\bibitem[VSF00]{v-s-f}
V.\ Voevodsky, A.\ Suslin and E.\ M.\ Friedlander {\em Cycles, transfers and motivic homology theories}, Ann.\ of Math.\ Studies, vol \textbf{143}, Princeton Univ. Press, 2000.



\bibitem[Voe02]{voe}
V.\ Voevodsky, {\em Motivic cohomology groups are isomorphic to higher Chow groups in any characteristic}, Int.\ Math.\ Res.\ Not.\ \textbf{7} (2002), 351--355.

\bibitem[Voe03]{Voe-milnor}
V.\ Voevodsky, {\em Motivic cohomology with $\Z/2$-coefficients}, Publ.\ Math.\ IHES \textbf{98} (2003), 59--104.

\bibitem[Voe11]{Voevodsky}
V.\ Voevodsky, {\em On motivic cohomology with $Z/l$-coefficients}, Ann.\ of Math.\ \textbf{174} (2011), 401--438.


\bibitem[Voi12]{Voi-unramified}
C.\ Voisin, {\em Degree $4$  unramified cohomology with finite coefficients and torsion codimension  $3$ cycles}, in Geometry and Arithmetic, (C.\ Faber, G.\ Farkas, R.\ de Jong Eds), Series of Congress Reports, EMS 2012, 347--368.

\end{thebibliography}
\end{document}